\documentclass[11pt,oneside]{article}
\usepackage[utf8]{inputenc}
\usepackage[english]{babel}
\usepackage{amsmath}
\usepackage{amsfonts}
\usepackage{amssymb}
\usepackage{amsthm}
\usepackage{tikz-cd} 
\usepackage[left=3cm,right=3cm,top=3cm,bottom=3cm]{geometry}

\theoremstyle{plain}
\newtheorem{teo}{}[section]
\newtheorem{prop}[teo]{Proposition}
\newtheorem{cor}[teo]{Corollary}
\newtheorem{rem}[teo]{Remark}
\newtheorem{lem}[teo]{Lemma}
\newtheorem{thm}[teo]{Theorem}
\newtheorem{df}[teo]{Definition}
\theoremstyle{definition}
\newtheorem{ex}[teo]{Example}

\newcommand\blfootnote[1]{%
  \begingroup
  \renewcommand\thefootnote{}\footnote{#1}%
  \addtocounter{footnote}{-1}%
  \endgroup
}

\title{On some topological realizations of groups and homomorphisms}
\author{Pedro J. Chocano, Manuel A. Mor\'on and Francisco R. Ruiz del Portal}
\date{}

\begin{document}
\maketitle

\begin{abstract}
Let $f:G\rightarrow H$ be a homomorphism of groups. We construct a topological space $X_f$ such that its group of homeomorphisms is isomorphic to $G$, its group of homotopy classes of self-homotopy equivalences is isomorphic to $H$ and the natural map between the group of homeomorphisms of $X_f$ and the group of homotopy classes of self-homotopy equivalences of $X_f$ is $f$. In addition, we consider realization problems involving homology, homotopy groups and groups of automorphisms.
\end{abstract}

\section{Introduction}
\blfootnote{2020  Mathematics  Subject  Classification: 20B25, 06A06, 06A11 05E18, 55P10, 55P99.}
\blfootnote{Keywords: Automorphism group, homotopy equivalence, Alexandroff spaces, posets, homology groups, homotopy groups.}
\blfootnote{This research is partially supported by Grants PGC2018-098321-B-100 and BES-2016-076669 from Ministerio de Ciencia, Innovación y Universidades (Spain).}
%\blfootnote{\textbf{Declaration:} On behalf of all authors, the corresponding author states that there is no conflict of interest. }

Alexandroff spaces are topological spaces with the property that the arbitrary intersection of open sets is open. That sort of topological spaces was first studied by P.S. Alexandroff in \cite{alexandroff1937diskrete}, where it is shown that they can also be seen as partially ordered sets. This viewpoint can be used to express topological notions in combinatorial terms. A particular case of Alexandroff spaces are finite topological spaces. There are two foundational papers on this subject that were published independently in 1966 \cite{mccord1966singular,stong1966finite}. In \cite{stong1966finite}, R.E. Stong made an analysis of the homeomorphism classification of finite topological spaces using matrices and also introduced combinatorial techniques to study their homotopy type. In \cite{mccord1966singular}, M.C. McCord studied the singular homology groups and homotopy groups of Alexandroff spaces proving that for every Alexandroff space $X$ there exists a simplicial complex $\mathcal{K}(X)$ sharing the same homotopy groups and singular homology groups. In fact, it is shown that there is a continuous map $f_X:|\mathcal{K}(X)|\rightarrow X$ inducing isomorphism on all homotopy groups. The converse result is also obtained, that is, given a simplicial complex $L$, there exists an Alexandroff space $\mathcal{X}(L)$ having the same singular homology groups and homotopy groups of $L$. %There are two recent monographs, that treat the algebraic aspects of finite topological spaces, due to J.P. May and J.A Barmak, \cite{may1966finite} and \cite{barmak2011algebraic} respectively. 

Finite topological spaces or Alexandroff spaces are a good tool to solve realization problems: given a category $C$ and a group $G$, is there an object $X$ in $C$ such that the group of automorphisms of $X$ is isomorphic to $G$? In \cite{birkhoff1936order,thornton1972spaces,barmak2009automorphism}, it is proved that for every finite group $G$ there exists a finite topological space $X_G$ such that its group of homeomorphisms is isomorphic to $G$. Recently, in \cite{barmak2020automorphism2}, J.A. Barmak constructed a finite topological space $X_G$ with $4|G|$ points and lower cardinality than the finite topological spaces obtained in \cite{birkhoff1936order,thornton1972spaces,barmak2009automorphism}. On the other hand, L. Babai in \cite{babai1980finite} obtained a finite topological space $X_G$ with $3|G|$ points realizing $G$ as a group of homeomorphisms. This topological space has the disadvantage that it requires to find a ``good'' set of generators of $G$ satisfying a list of non-trivial properties. A generalization of these results for non-finite groups was made in \cite{chocano2020topological}, where the realization problem for the homotopical category $HTop$ and the pointed homotopical category $HTop_*$ was also solved. To be precise, it was proved that for every group $G$ there exists an Alexandroff space $\overline{X}_G^*$ such that its group of automorphisms in the topological category $Top$, $HTop$ and $HTop_*$ is isomorphic to $G$. 

The restriction of $HTop$ to topological spaces with the homotopy type of a CW-complex is denoted by $HPol$. In $HPol$ and the pointed version $HPol_*$, the realizability problem has a long history, see for instance \cite{kahn1976realization}. Recently, C. Costoya and A. Viruel solved the realization problem in $HPol_*$ for finite groups in \cite{costoya2014every}.
%($Aut(\overline{X}_G^*)$) , ($\mathcal{E}(\overline{X}_G^*)$)  ($\mathcal{E}_*(\overline{X}_G^*)$)

We introduce a bit of notation. Let $X$ be a topological space. Let us denote by $Aut(X)$ the group of homeomorphisms of $X$. Let $\mathcal{E}(X)$ denote the group of homotopy classes of self-homotopy equivalences of $X$. Let $\mathcal{E}_*(X)$ denote the group of pointed homotopy classes of pointed self-homotopy equivalences of $X$.

The following result answers a natural question posed by professor Jesús Antonio Álvarez López during a talk in the VIII Encuentro de Jóvenes Topólogos (A Coruña, 2019).

\begin{lem}\label{thm:teoremaPrincipal} Let $G$ and $H$ be two groups. There exists a topological space $X_H^G$ such that $Aut(X_H^G)$ is isomorphic to $G$ and $\mathcal{E}(X_H^G)$ is isomorphic to $H$.
\end{lem}
An immediate consequence of Lemma \ref{thm:teoremaPrincipal} is the following: for a general topological space there is no relation between its group of homeomorphisms and its group of homotopy classes of self-homotopy equivalences. Moreover, if $G$ and $H$ are two groups, then we can produce infinitely many non-homeomorphic topological spaces having $G$ as their group of homeomorphisms and having $H$ as their group of homotopy classes of self-homotopy equivalences.

For every topological space $X$ there is a natural homomorphism of groups $\tau: Aut(X)\rightarrow \mathcal{E}(X)$ sending each homeomorphism $f$ to its homotopy class $[f]$. Given two groups $G$ and $H$, we consider the topological space $X_G^H$ obtained in Lemma \ref{thm:teoremaPrincipal}. The kernel of $\tau:Aut(X_G^H)\rightarrow \mathcal{E}(X_G^H)$ corresponds precisely to $Aut(X_H^G)$. The image of $\tau$ is the homotopy class of the identity map. However, modifying the construction of the topological space $X_H^G$, we can obtain a stronger version of Lemma \ref{thm:teoremaPrincipal}.
\begin{thm}\label{thm:realizacionHomo}
Let $f:G\rightarrow H$ be a homomorphism of groups. There exists a topological space $X_f$ such that $Aut(X_f)=G$, $\mathcal{E}(X_f)=H$ and $\tau=f$.
\end{thm}
Theorem \ref{thm:realizacionHomo} clearly generalizes Lemma \ref{thm:teoremaPrincipal}. We prefer to prove Lemma \ref{thm:teoremaPrincipal} first for the sake of exposition. Omitting Lemma \ref{thm:teoremaPrincipal}, the proof of Theorem \ref{thm:realizacionHomo} becomes less intuitive. In addition, subsequent results are obtained using the topological space given in Lemma \ref{thm:teoremaPrincipal}.

%Modifying the topological space constructed in Theorem \ref{thm:teoremaPrincipal}, it can be obtained a stronger version of that theorem.
%
%\begin{thm}\label{thm:SegundoPrincipal} Let $f:G\rightarrow H$ be a homomorphisms of groups, there exists a topological space $X$ such that $Aut(X)=G$, $\mathcal{E}(X)=H$ and the natural homomorphism of groups $\tau:Aut(X)\rightarrow \mathcal{E}(X)$ is $f$.
%\end{thm}

Given a topological space $X$ and a non-negative integer number $n$, we can consider the $n$-th homology group $H_n(X)$ of $X$ or the $n$-th homotopy group $\pi_n(X)$ of $X$. It is natural to consider more realization problems involving these groups, the group of homeomorphisms and the group of homotopy classes of self-homotopy equivalences.

%For a topological space $X$ and a non-negative integer number $n$, it can be considered the $n$-th homology group $(H_n(X))$ or the $n$-th homotopy group $(\pi_n(X))$ to state realizability problems. In the following results, we  combine at the same time some of the realizability problems introduced previously. 
%\begin{thm}\label{thm:homologygroupsAutE} Let $G$ be a group, let $H$ be a finite group and let $X$ be a topological space with the homotopy type of a compact CW-complex. There exists an Alexandroff space $\overline{X}_H^G$ such that $Aut(\overline{X}_H^G)$ is isomorphic to $ G$, $\mathcal{E}(\overline{X}_H^G)$ is isomorphic to $ H$, $H_n(\overline{X}_H^G) $ is isomorphic to $  H_n(X)$ and $\pi_n(\overline{X}_H^G)$ is isomorphic to $\pi_n(X)$ for every $n\in \mathbb{N}$.
%\end{thm}
\begin{thm}\label{thm:homologygroupsAutE} Let $G$ and $H$ be finite groups and let $X$ be a topological space with the homotopy type of a compact CW-complex. There exists an Alexandroff space $\overline{X}_H^G$ such that $Aut(\overline{X}_H^G)$ is isomorphic to $ G$, $\mathcal{E}(\overline{X}_H^G)$ is isomorphic to $H$ and $X$ is weak homotopy equivalent to $\overline{X}_H^G$, which implies that $H_n(\overline{X}_H^G) $ is isomorphic to $  H_n(X)$ and $\pi_n(\overline{X}_H^G)$ is isomorphic to $\pi_n(X)$ for every $n\in \mathbb{N}$.
\end{thm}
As an immediate consequence of Theorem \ref{thm:homologygroupsAutE}, we can deduce the following corollaries.
%\begin{cor}\label{cor:sequenceOfGroups} Let $H$ be a finite group, let $G$ be a group and let $\{F_{i} \}_{i\in I}$ be a sequence of finitely generated abelian groups, where $I\subset \mathbb{N}$ is a finite set. Then, there exists a topological space $X$ such that $Aut(X)$ is isomorphic to $G$, $\mathcal{E}(X)$ is isomorphic to $ H$ and $H_{i}(X)$ is isomorphic to $F_{i}$ for every  $i\in I$.
%\end{cor} 
\begin{cor}\label{cor:sequenceOfGroups} Let $H$ and $G$ be finite groups and let $\{F_{i} \}_{i\in I}$ be a set of finitely generated Abelian groups, where $I\subset \mathbb{N}$ is a finite set. There exists a topological space $X$ such that $Aut(X)$ is isomorphic to $G$, $\mathcal{E}(X)$ is isomorphic to $ H$ and $H_{i}(X)$ is isomorphic to $F_{i}$ for every  $i\in I$.
\end{cor}

\begin{cor}\label{cor:homotopyGroupsandAuotomorphis} Let $H$ and $G$ be finite groups,  $n\in \mathbb{N}$ and let $T$ be a finitely presented (Abelian) group (if $n>1$). There exists a topological space $X$ such that $Aut(X)$ is isomorphic $G$, $\mathcal{E}(X)$ is isomorphic to $H$ and $\pi_n(X)$ is isomorphic to $T$.  
\end{cor}
%
%\begin{cor}\label{cor:homotopyGroupsandAuotomorphis} Let $H$ be a finite group, let $G$ be a group,  $n\in \mathbb{N}$ and let $T$ be a finitely presented (abelian) group (if $n>1$). There exists a topological space $X$ such that $Aut(X)$ is isomorphic $G$, $\mathcal{E}(X)$ is isomorphic to $H$ and $\pi_n(X)$ is isomorphic to $T$.  
%\end{cor}

Roughly speaking, these corollaries say that for a general topological space $X$ its group of automorphisms in $Top$ or $HTop$ does not have any relation to its $n$-th homology or homotopy group and vice versa. In contrast, for the category $HPol$, the situation is completely different since $\mathcal{E}(X)$ contains normal subgroups that are nilpotent. For instance, given a topological space $X$, we denote by $\mathcal{E}_\#(X)$ $(\mathcal{E}_*(X))$ the set of self-homotopy equivalences that induce the identity map in homotopy (homology). It is trivial to check that $\mathcal{E}_\#(X)$ $(\mathcal{E}_*(X))$ is a normal subgroup of $\mathcal{E}(X)$ and if $X$ is a finite $CW$-complex, then $\mathcal{E}_\#(X)$ $(\mathcal{E}_*(X))$ is a nilpotent group. See \cite{costoya2020primer} for more details. Using the construction obtained in Theorem \ref{thm:homologygroupsAutE}, we can find topological spaces that do not satisfy the previous properties. From this we get that some of the techniques used to study the group of self-homotopy equivalences for $CW$-complexes cannot be adapted in a natural way to general spaces.

The organization of the paper is as follows. In Section \ref{section:preliminaries} we introduce basic concepts and results from the literature. In Section \ref{sec:example} we provide an example of one of the main results in order to motivate the main ideas of the proof of Lemma \ref{thm:teoremaPrincipal}. In Section \ref{sec:proofMainTheorem} we prove Lemma \ref{thm:teoremaPrincipal} and give some remarks.
In Section \ref{section:pruebaSegundoTeorema} we prove Theorem \ref{thm:realizacionHomo}. In Section \ref{section:homologyHomotopy} we define a sequence of topological spaces whose homotopy and homology groups are all trivial, their group of automorphisms in $Top$ and $HTop$ are also trivial, but they are not homeomorphic or homotopy equivalent to a point. Then, we prove Theorem \ref{thm:homologygroupsAutE} as well as Corollary \ref{cor:sequenceOfGroups} and Corollary \ref{cor:homotopyGroupsandAuotomorphis}. Finally, we give examples of topological spaces satisfying that the groups $\mathcal{E}_*(\cdot)$ and $\mathcal{E}_\#(\cdot)$ are not nilpotent in general.

%Despite the fact that Lemma \ref{thm:teoremaPrincipal} can be seen as an immediate corollary of Theorem \ref{thm:realizacionHomo}, for the sake of clarity we prefer to introduce firstly Lemma \ref{thm:teoremaPrincipal} and then prove Theorem \ref{thm:realizacionHomo} in Section \ref{section:pruebaSegundoTeorema}.

\section{Preliminaries}\label{section:preliminaries}
The following definitions and results can be found with more detail in \cite{alexandroff1937diskrete,barmak2011algebraic,mccord1966singular,stong1966finite,may1966finite} .

\begin{df} Let $X$ and $Y$ be topological spaces. A continuous function $f:X\rightarrow Y$ is said to be a weak homotopy equivalence if it induces isomorphisms on all the homotopy groups.
\end{df}
\begin{df} An Alexandroff space $X$ is a topological space satisfying that the arbitrary intersection of open sets is open. 
\end{df}
If $X$ is an Alexandroff space, then for every $x\in X$ there exists a minimal open neighbourhood $U_x$ given by the intersection of every open set containing $x$. $F_x$ denotes the set given by the intersection of every closed set containing $x$. Trivially, every finite topological space is an Alexandroff space. An Alexandroff space $X$ is locally finite if for every $x\in X$ the set $U_x$ is finite.

Let $(X,\leq)$ be a partially ordered set or poset. If $x,y\in X$, then we write $x\prec y$ ($x\succ y$) if and only if $x<y$ ($x>y$) and there is no $z\in X$ such that $x<z<y$ ($x>z>y$). We will denote by $max(X)$ the maximum of $X$ if it exists. We denote by $P_x=(E_x,S_x)$ the cardinal numbers $E_x=|\{y\in X|y\prec x\}|$ and $S_x=|\{ y\in X|x\prec y\}|$. A set $S\subseteq X$ is called lower (upper) if for every $x\in S$ and $y\leq x$ ($y \geq x$) we have $y\in S$.

It is not difficult to verify the following two properties:
\begin{itemize}
\item For a partially ordered set $(X,\leq)$ the family of lower (upper) sets of $\leq$ is a $T_0$ topology on $X$, that makes $X$ a $T_0$ Alexandroff space.
\item For a $T_0$ Alexandroff space, the relation $x\leq_{\tau} y$ if and only if $U_x\subset U_y$ ($U_y \subset U_x$) is a partial order on $X$. 
\end{itemize}

In addition, for a set $X$, the $T_0$ Alexandroff space topologies on $X$ are in bijective correspondence with the partial orders on $X$.

From now on, every Alexandroff space satisfies the $T_0$ separation axiom. The following results can be found, for instance, in \cite{barmak2011algebraic,may1966finite}.
\begin{prop} If $f:X\rightarrow Y$ is a map between Alexandroff spaces, then $f$ is continuous if and only if $f$ preserves the order.
\end{prop}
From this and previous properties, the following result can be deduced.
\begin{thm}
The category of $T_0$ Alexandroff spaces is isomorphic to the category of partially ordered sets.
\end{thm}
Hence, partially ordered sets and $T_0$ Alexandroff spaces can be treated as the same object.
\begin{prop} Let $f,g:X\rightarrow Y$ be continuous maps between Alexandroff spaces. If $f(x)\leq g(x)$ ($f(x)\geq g(x)$) for every $x\in X$, then $f$ and $g$ are homotopic.
\end{prop}
\begin{rem}\label{rem:minimumiscontractible} If $X$ is a $T_0$ Alexandroff space with a minimum (maximum) $x*$, then $X$ is contractible to $x*$. This follows from the previous proposition and the fact that the constant map $c:X\rightarrow X$ given by $c(x)=x*$ satisfies that $c(x)\leq x$ ($c(x)\geq x$) for every $x\in X$. 
\end{rem}

%CAMBIAR Y ARREGLAR PARA MEJORAR Y USAR EN LA DEMOSTRACION DEL LEMA 5.4, definicion  y proposicion siguientes
\begin{df} Given a finite poset $(X,\leq)$, the height $ht(X)$ of $X$ is one less than the maximum number of elements in a chain of $X$. The height of a point $x$ in a locally finite Alexandroff space is given by $ht(U_x)$. For a general Alexandroff space $X$, the height of a point $x\in X$ is defined as $\infty$ if $U_x$ contains a chain without a minimum and $ht(U_x)$ otherwise.% The height to a minimal point $y\in X$, $sht(X,y)$ is one less the minimum number of elements in a chain of $X$ containing $y$. The height of a point $x$ in an Alexandroff space to a minimal point $y$ is given by $sht(U_x,y)$.
\end{df}
\begin{ex} Let us consider the real numbers with the usual order. For every $x\in \mathbb{R}$ the height of $x$ is $\infty$ because the chain $...<x-n<...<x-2<x-1<x$ does not have a minimum. Moreover, $P_x=(0,0)$ since there is no $y\in \mathbb{R}$ satisfying that $x\prec y$ or $x\succ y$. Let us consider $X=\mathbb{N}\cup \{* \}$, where we consider the partial order defined as follows: $n<*$ for every $n\in \mathbb{N}$. It is clear that $U_*$ is not a finite set but the height of $*$ is $1$. Furthermore, $P_*=(|\mathbb{N}|, 0)$.
\end{ex}

\begin{prop}\label{prop:sucesores} Let $X$ and $Y$ be Alexandroff spaces. If $f:X\rightarrow Y$ is a homeomorphism and $x\prec y$, then $f(x)\prec f(y)$ ($f(x)\succ  f(y)$). Furthermore, for every $x\in X$ the height of $x$ is equal to the height of $f(x)$ and $P_x=P_{f(x)}$.
\end{prop}
%\begin{prop} If $X$ is a finite $T_0$ topological space and $f:X\rightarrow X$ is a homeomorphism. Then, $f(ht(U_x))=ht(U_{f(x)})$ for every $x\in X$. If $y\in X$ is a minimal point. Then,  $f(ht(U_x,y))=ht(U_{f(x)},f(y))$.
%\end{prop}

The following results provide a combinatorial way to study the homotopy and weak homotopy type of finite topological spaces.
\begin{df} Let $X$ be an Alexandroff space. A point $x$ in $X$ is a down beat point (resp. up beat point) if $U_x\setminus\{x \}$ has a maximum (resp. $F_x\setminus \{x \}$ has a minimum). A finite $T_0$ topological space $X$ is a minimal finite space if it has no beat points. A core of a finite topological space is a strong deformation retract which is a minimal finite space. 
\end{df}

%The following proposition was enunciated originally for finite topological spaces but it can be adapted to Alexandroff spaces.
\begin{prop}[\cite{stong1966finite}]\label{prop:quitarbeatpoints} Let $X$ be an Alexandroff space and let $x\in X$ be a beat point. Then $X\setminus \{ x\}$ is a strong deformation retract of $X$.
\end{prop} 
If $X$ is a finite $T_0$ topological space, then $X$ has a core. We only need to remove beat points one by one to obtain a minimal finite space.
\begin{thm}[\cite{stong1966finite}] \label{thm:stonghomeoeshomoeq} If $X$ is a minimal finite space, then $f:X\rightarrow X$ is a homeomorphism if and only if $f$ is a homotopy equivalence.
\end{thm}

\begin{cor}\label{cor:autesigualaE} If $X$ is a minimal finite space, then $Aut(X)$ is isomorphic to $\mathcal{E}(X)$.
\end{cor}
\begin{rem}\label{rem:aut=e}
The result of Corollary \ref{cor:autesigualaE} can be stated in a stronger way. Let $X$ be a minimal finite space. It is easy to check that  $Aut(X)=\mathcal{E}(X)$. For every $[f]$ in $\mathcal{E}(X)$ there is only one element in the class $[f]$. We can identify every homeomorphism with its homotopy class.
\end{rem}
\begin{rem}\label{rem:generalizarKukiela}
Corollary \ref{cor:autesigualaE} can be generalized to Alexandroff spaces. To do this, consider the notion of being locally a core introduced in \cite{kukiela2010homotopy}. This notion generalizes the notion of minimal finite space, that is, every minimal finite space is locally a core. Let $X$ be locally a core. Then a continuous map $f:X\rightarrow X$ is a homeomorphism if and only if $f$ is a homotopy equivalence. Moreover, we get that $Aut(X)\simeq \mathcal{E}(X)$.
\end{rem}

\begin{df} Let $X$ be a finite $T_0$ topological space. A point $x$ in $X$ is a down weak beat point (resp. up weak beat point) if $U_x\setminus\{x \}$ is contractible (resp. $F_x\setminus \{x \}$ is contractible).
\end{df}
\begin{prop}[\cite{barmak2011algebraic}] Let $X$ be a finite $T_0$ topological space and let $x\in X$ be a weak beat point. Then the inclusion $i:X\setminus \{ x\}\rightarrow X$ is a weak homotopy equivalence.
\end{prop}
\begin{df}[\cite{barmak2011algebraic}] Let $X$ be a finite $T_0$ topological space and let $Y\subset X$. It is said that $X$ collapses to $Y$ by an elementary collapse if $Y$ is obtained from $X$ by removing a weak beat point. Given two finite $T_0$ topological spaces $X$ and $Y$, $X$ collapses to $Y$ if there is a sequence $X=X_1,X_2,...,X_n=Y$ of finite $T_0$ topological spaces such that for each $1\leq i<n$, $X_i$ collapse to $X_{i+1}$ by an elementary collapse. %A finite $T_0$ topological space is collapsible to a point if $X$ collapse to a point.
\end{df}

\begin{lem}\label{lem:beatPoints} Let $X$ and $Y$ be Alexandroff spaces and let $f:X\rightarrow Y$ be a homeomorphism. Then $x\in X$ is a down (up) weak beat point if and only if $f(x)$ is a down (up) weak beat point.
\begin{proof}
There is no loss of generality in assuming that $x$ is a down weak beat point. If $x$ is an up weak beat point, then the argument is similar. It is easy to see that $f(U_x)=U_{f(x)}$. %$f(U_x)\subset U_{f(x)}$ holds trivially. If $y\in U_{f(x)}$, $y\leq f(x)$ so $f^{-1}(y)\leq x$, from here, we can deduce that $y\in f(U_x)$.  
Therefore, $f(U_x\setminus \{x\})=U_{f(x)}\setminus \{f(x)\}$ and we get the desired result.

%The other implication is analogue, suppose that $f(x)$ is a down beat point. Then, there exists $y\in Y$ with $y<f(x)$ such that for every $z\in Y$ satisfying $z<f(x)$  we get $z\leq y $. Let us take $f^{-1}(y)$. We want to show that for every $w\in X$ with $w<x$ we have $w\leq f^{-1}(y)$. By continuity of $f$, $f(w)<f(x)$. By hypothesis, $f(w)\leq y$. Therefore, $w\leq f^{-1}(y)$ as we wanted.

\end{proof}

\end{lem}

We recall the notion of a Hasse diagram for a locally finite Alexandroff space $X$. The Hasse diagram $H(X)$ of $X$ is a directed graph. The vertices of $H(X)$ are the points of $X$. There is an edge between two vertices $x$ and $y$ if and only if $x\prec y$ and the orientation of the edge is from the lower element to the upper element. We omit the orientation of the subsequent Hasse diagrams and we assume an upward orientation.

\begin{rem} It is easy to identify beat points of a finite topological space $X$ by looking at its Hasse diagram. A vertex $x$ is a down beat point (resp. up beat point) if there is only one edge that enters (exits) it, i.e.,  $P_x=(a,b)$, where $a=1$ ($b=1$).
\end{rem}

The homotopy and singular homology groups of Alexandroff spaces were studied in \cite{mccord1966singular}.

\begin{df} Let $X$ be an Alexandroff space. Its McCord complex or order complex $\mathcal{K}(X)$ is the simplicial complex whose simplices are the non-empty chains of $X$. Let $L$ be a simplicial complex. The face poset of $L$, denoted by $\mathcal{X}(L)$, is defined to be the poset of simplices of $L$ ordered by inclusion.
\end{df}
\begin{rem} A finite $T_0$ topological space $X$ is said to be collapsible if it collapses to a point. If $X$ is a collapsible finite $T_0$ topological space, then $\mathcal{K}(X)$ is also collapsible.
\end{rem}
The geometric realization of a simplicial complex $K$ is denoted by $|K|$.

\begin{thm}\cite{mccord1966singular} Given an Alexandroff space $X$, there exists a weak homotopy equivalence $f:|\mathcal{K}(X)|\rightarrow X$.
\end{thm}

\begin{thm}\cite{mccord1966singular} Given a simplicial complex $L$, there exists a weak homotopy equivalence $f:|L|\rightarrow \mathcal{X}(L)$.
\end{thm}

Finally, we recall some remarks and a definition. For a more complete treatment we refer to the reader to \cite{barmak2011algebraic}.
\begin{df}\label{df:non-hausdorffsuspension} The non-Hausdorff join $X\circledast Y$ of two Alexandroff spaces  $X$ and $Y$ is the disjoint union $X\sqcup Y$ keeping the given ordering within $X$ and $Y$ and setting $y\leq x$ for every $x\in X$ and $y\in Y$.
\end{df}
\begin{rem}\label{rem:nonhausdorffsuspensioncontractible} Given two Alexandroff spaces $X$ and $Y$, $\mathcal{K}(X\circledast Y)= \mathcal{K}(X)\ast \mathcal{K}(Y)$, where $\ast$ denotes the usual join of simplicial complexes. If $X$ and $Y$ are finite topological spaces and one of them is collapsible, then $\mathcal{K}(X\circledast Y)$ is collapsible.
\end{rem}
\begin{prop}\label{non-hausdorffsuspensionAut} Let $X$ and $Y$ be two Alexandroff spaces. Then $Aut(X\circledast Y)=Aut(X)\times Aut(Y)$.
\end{prop}

\section{Examples and motivation of the proof of Lemma \ref{thm:teoremaPrincipal}}\label{sec:example}
We present an example to illustrate the idea of the construction given in the proof of Lemma \ref{thm:teoremaPrincipal}.

\begin{ex}\label{ex:primerEjemplo} Let us consider the Klein four-group $\mathbb{Z}_2\oplus \mathbb{Z}_2$, where we denote  $g_1=(0,0), g_2=(1,0),g_3=(0,1)$ and $g_4=(1,1)$, and the cyclic group of two elements $\mathbb{Z}_2$, where we denote $h_1=0$ and $h_2=1$. We also denote $G=\mathbb{Z}_2\oplus \mathbb{Z}_2$ and $H=\mathbb{Z}_2$ for simplicity.  Moreover, let $S_{G}=\{ g_2, g_3\}$, $S_{H}=\{h_2 \}$ be generating sets of $G,H$ respectively. We declare $g_2<g_3$. Our goal is to find a finite $T_0$ topological space $X_H^G$ such that $Aut(X_H^G)\simeq G$ and $\mathcal{E}(X_H^G)\simeq H$.

By \cite{barmak2009automorphism,chocano2020topological}, there exists a finite $T_0$ topological space $X^G$ satisfying that $Aut(X^G)$ is isomorphic to $G$. In Figure \ref{fig:EjemploSeparado1} we have represented in blue the Hasse diagram of $X^G$. It is clear that adding to $X^G$ a minimum $*$, i.e., $X_*^G=X^G \cup \{* \}$ with $*<x$ for every $x\in X_G$, we get that $\mathcal{E}(X^G_*)$ is trivial since $X_*^G$ is contractible. On the other hand, if $f\in Aut(X^G_*)$, then we have that $f(*)=*$ because $*$ is a minimum. Thus we deduce that $Aut(X^G_*)$ is isomorphic to $Aut(X^G)$. Our next goal is to find a topological space $X_H^*$ satisfying that $Aut(X_H^*)$ is trivial and $\mathcal{E}(X_H^*)$ is isomorphic to $ H$. Again, by \cite{chocano2020topological}, there exists a finite $T_0$ topological space $X_H$ satisfying that $\mathcal{E}(X_H)\simeq Aut(X_H)\simeq H$. In Figure \ref{fig:EjemploSeparado1}, the Hasse diagram of $X_H$ corresponds to the red and black parts of the diagram on the right. We modify $X_H$ in order to reduce the number of self-homeomorphisms without changing the number of self-homotopy equivalences. For this purpose we add some points to $X_H$. We consider $X_H^*=X_H\cup \{w_{h_1},w_{h_2},a\}$, where we have the following relations: $A_{(h_1,0)}\prec w_{h_1}$ and $A_{(h_2,0)}\prec w_{h_2}\prec a$. The Hasse diagram of $X_H^*$ can be seen on the right on Figure \ref{fig:EjemploSeparado1}, where the new points are pictured in orange. It is easy to check that $\mathcal{E}(X_H^*)\simeq \mathcal{E}(X_H)\simeq H$. The new points are beat points so we can remove them without changing the homotopy type of $X_H^*$. Hence, we have that $X_H^*$ and $X_H$ have the same homotopy type. We prove that $Aut(X_H^*)$ is trivial. If $f\in Aut(X_H^*)$, then $f(M)=M$ and $f(N)=N$, where $M$ and $N$ denote the set of maximal and minimal elements of $X_H^*$ respectively. From this, using Proposition \ref{prop:sucesores} and the fact that $f$ preserves heights, we deduce that $f$ is the identity.

% and its dual $X_G^H$ for $H=\mathbb{Z}_2$ and $G=D_2$. 

%We construct the spaces $X_*^G$ and $X_H^*$, we represent them by the Hasse diagram of Figure \ref{fig:EjemploSeparado1}. We have omitted some subscripts for simplicity.

\begin{figure}[h]\center
\includegraphics[scale=0.82]{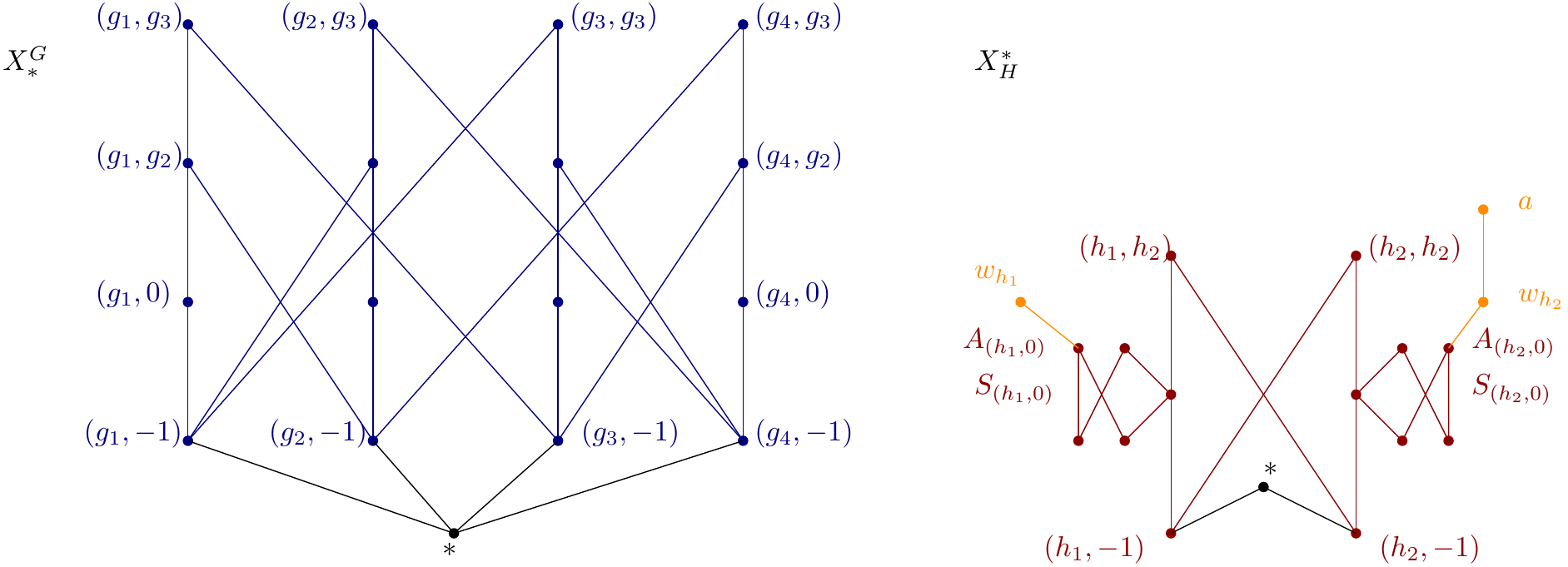}
\caption{Hasse diagrams of $X_*^G$ and $X_H^*$.}
\label{fig:EjemploSeparado1}
\end{figure}

Combining $X_H^*$ with $X_*^G$ we obtain $X_H^G$. We identify the point $*$ of $X_H^*$ and $X_*^G$. We also extend the partial order of the two previous posets using transitivity, that is, if $x\in X_H^*$ and $y\in X_*^G$, then we have that $x< y$ if and only if $x<*<y$. It is not difficult to check that $X_H^G$ satisfies the properties required at the beginning. $X_H^G$ and $X_H$ have the same homotopy type because we can collapse $X_*^G$ to $*$. We thus get that $\mathcal{E}(X_H^G)\simeq \mathcal{E}(X_H)\simeq H$. Since $*$ is the only point with height $1$ and $P_*=(2,4)$, it follows that $*$ is a fixed point for every $f\in Aut(X_H^G)$. By the continuity of $f$, it is easily seen that $f(X^G_*)=X^G_*$ and $f(X_H^*)=X_H^*$. From this, we conclude that $Aut(X_H^G)$ is isomorphic to $G$.
%Finally, we identify both spaces in $\{ *\}$ so as to obtain $X_H^G$. We present the Hasse diagram of $X_H^G$ in Figure \ref{fig:Ejemplojunto}. In this example, following the arguments from the proof of Theorem \ref{thm:TeoremaParaUnEspacioyFinitos}, it is trivial to check that $Aut(X_H^G)=G=D_4$ and $\mathcal{E}(X_{H}^G)=H=\mathbb{Z}_2$.
\begin{figure}[h]\center
\includegraphics[scale=0.82]{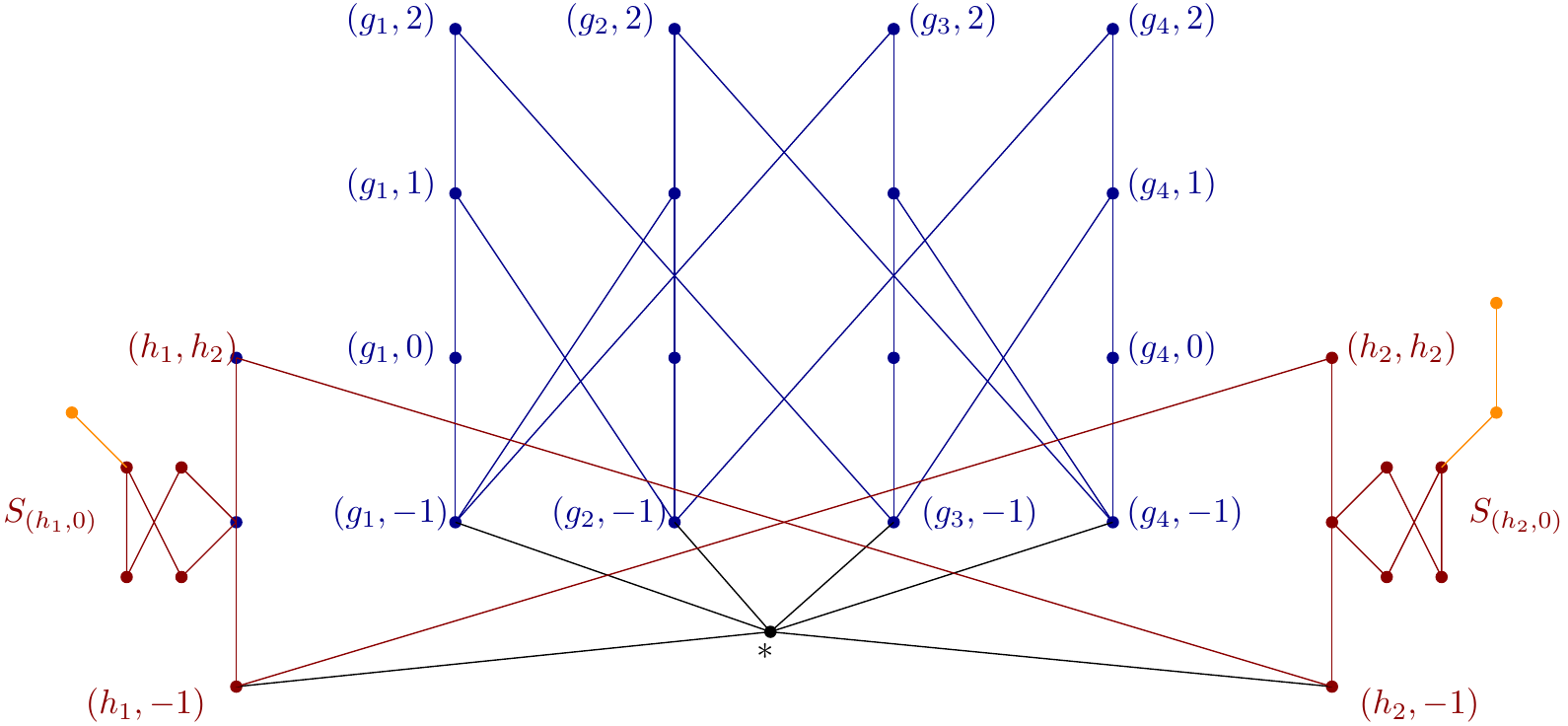}
\caption{Hasse diagram of $X_H^G$.}
\label{fig:Ejemplojunto}
\end{figure}

We can also consider what we call the dual case, that is, $X_G^H$, where we have that $Aut(X_G^H)\simeq H$ and $\mathcal{E}(X_G^H)\simeq G$. We can now proceed analogously to the previous arguments, i.e., we find $X_G^*$ and $X_*^H$. In Figure \ref{fig:Ejemploseparadodual} we have the Hasse diagrams of $X_*^H$ and $X_G^*$. 

\begin{figure}[h]\center
\includegraphics[scale=0.82]{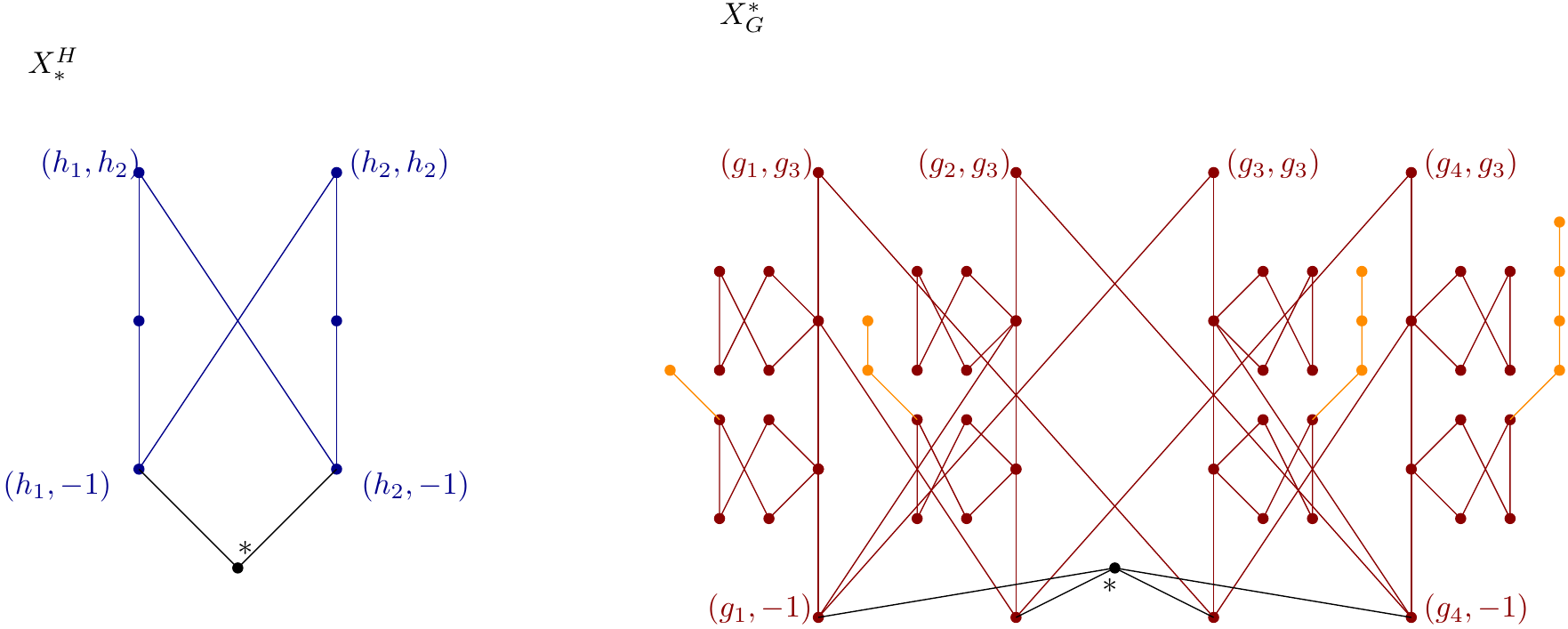}
\caption{Hasse diagrams of $X_*^H$ and $X_G^*$.}
\label{fig:Ejemploseparadodual}
\end{figure}
\newpage

Again, combining $X_*^H$ with $X_G^*$ we get $X_G^H$, that is, the finite topological space given by the Hasse diagram of Figure \ref{fig:Ejemplojuntodual}. It is trivial to verify that $X_G^H$ is not homeomorphic to $X_H^G$ because of their different cardinality. Furthermore, $X_H^G$ and $X_G^H$ are not homotopy equivalent since $\mathcal{E}(X_H^G)$ is not isomorphic to $\mathcal{E}(X_G^H)$. Another way to prove the last assertion is the following. After removing one by one the beat points of $X_H^G$ we get $X_H$; after removing one by one the beat points of $X_G^H$ we get $X_G$. However, $X_G$ is not homeomorphic to $X_H$ because of their different cardinality. By Theorem \ref{thm:stonghomeoeshomoeq} we conclude that $X_H^G$ and $X_G^H$ are not homotopy equivalent. Moreover, studying their McCord complexes it can be shown that $X_H^G$ and $X_G^H$ are not weak homotopy equivalent.
\begin{figure}[h]\center
\includegraphics[scale=0.82]{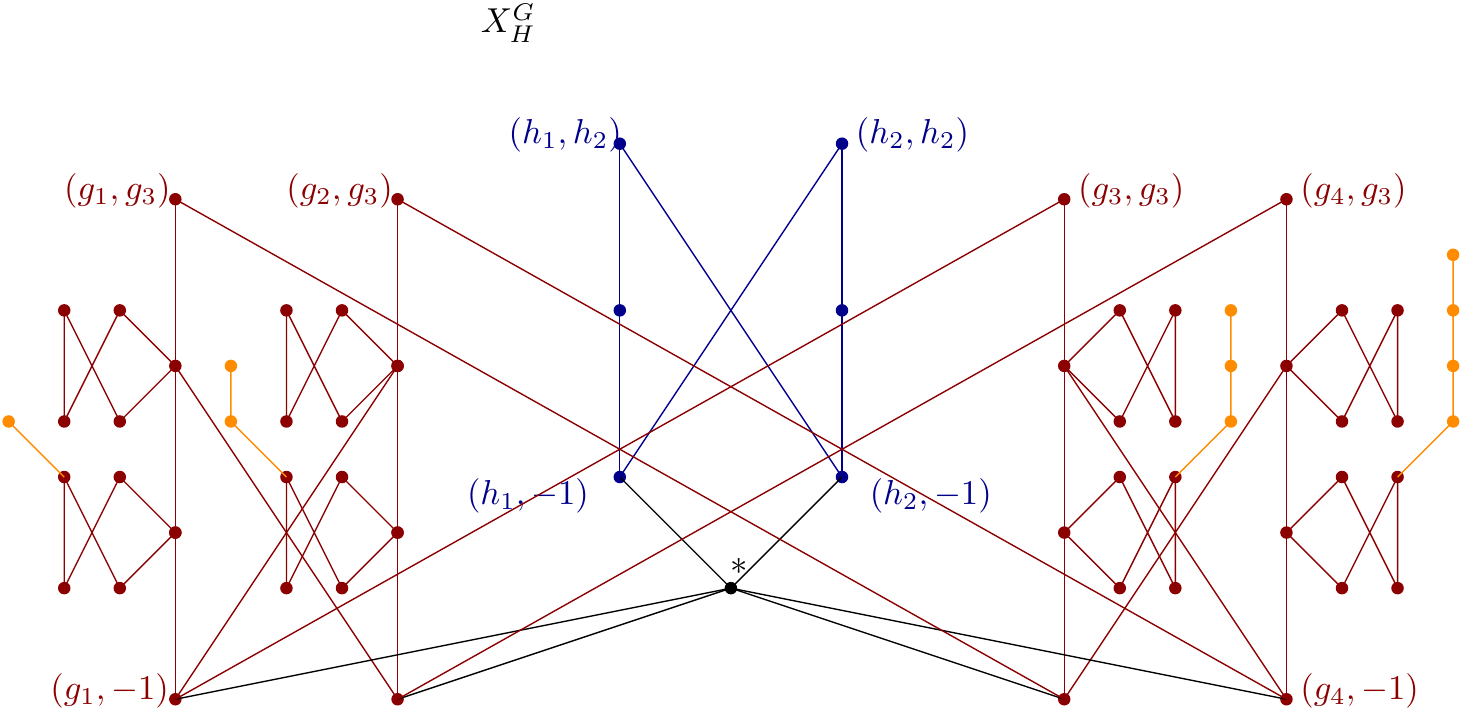}
\caption{Hasse diagram of $X_G^H$.}
\label{fig:Ejemplojuntodual}
\end{figure}

\end{ex}

\section{Proof of Lemma \ref{thm:teoremaPrincipal} and remarks}\label{sec:proofMainTheorem}
\begin{proof}[\textbf{Proof of Lemma \ref{thm:teoremaPrincipal}}]
We follow the same strategy as in Example \ref{ex:primerEjemplo}. Firstly, we find a topological space $X_H^*$ such that $\mathcal{E}(X_H^*)\simeq H$ and $Aut(X_H^*)$ is trivial. Secondly, we find a topological space $X_*^G$ such that $Aut(X_*^G)\simeq G$ and $\mathcal{E}(X_*^G)$ is trivial. Finally, we combine properly both topological spaces to obtain a topological space $X_H^G$ satisfying that $Aut(X_H^G)\simeq G$ and $\mathcal{E}(X_H^G)\simeq H$. 

We first assume that $G$ and $H$ are non-trivial groups. The trivial case will be considered later.

\underline{\textit{Construction of $X_*^G$ and properties.}} We consider a set of non-trivial generators $S'_G$  for $G$. Without loss of generality we can consider a well-order on $S'_G$ satisfying that if $max(S'_G)$ exists and $|S'_G|>1$, then there exists $\alpha\in S'_G$ satisfying $\alpha\prec max(S'_G)$. If $max(S'_G)$ exists and there is no $\alpha\in S'_G$ satisfying $\alpha\prec max(S'_G)$, then the well-order defined on $S'_G$ can be modified as follows: $max(S'_G)<\alpha$ for every $\alpha\in S'_G\setminus \{max(S'_G) \}$ and the rest of the relations defined on $S'_G\setminus \{max(S'_G) \}$ unaltered. It is obvious that the new partial order defined on $S'_G$ is indeed a well-order and $max(S'_G)$ does not exist. We consider $S_G=S'_G\cup \{0,-1 \}$, where we assume that $-1,0\notin S'_G$, and extend the well-order defined on $S'_G$ to $S_G$ as follows: $-1<0<\alpha$ for every $\alpha\in S'_G$. We consider 
$$X^G_*=(G\times S_G)\cup \{*\},$$
where we have the following relations:
\begin{itemize}
\item $(g,\alpha)<(g,\delta)$ if $1\leq \alpha<\delta$, where $g\in G$ and $\alpha,\delta\in S_G$.
\item $(g\alpha,-1)\prec (g,\alpha)$, where $g\in G$ and $\alpha\in S_G\setminus \{-1,0\}$.
\item $*\prec (g,-1)$, where $g\in G$.
\end{itemize}
The rest of the relations can be deduced from the above relations using transitivity. It is easy to check that $X_*^G$ is a partially ordered set.

We prove that $Aut(X_*^G)\simeq G$ and $\mathcal{E}(X_*^G)$ is the trivial group. We have that $\mathcal{E}(X_*^G)$ is the trivial group because $X_*^G$ is contractible to $*$, which is a minimum. Since $*$ is a minimum, it follows that every self-homeomorphism must fix this point. From this we deduce that $Aut(X_*^G)\simeq Aut(X_*^G\setminus \{* \})$. In addition, $X_*^G\setminus \{* \}$ is the same topological space considered in \cite[Section 3]{chocano2020topological} and denoted by $X_G$. Hence, we know that $Aut(X_*^G)\simeq Aut(X_G)\simeq G$, where $\varphi: G\rightarrow Aut(X_G)$ is given by $\varphi(s)(g,\alpha)=(sg,\alpha)$ and is an isomorphism of groups.

\underline{\textit{Construction of $X_H^*$ and properties.}} We consider a set of non-trivial generators $S'_H$ for $H$. There is no loss of generality in assuming that there exists a well-order on $S'_H$ satisfying that if $max(S'_H)$ exists and $|S'_H|>1$, then there exists $\alpha\in S_H$ satisfying $\alpha\prec max(S'_H)$. We repeat the same construction  made before, that is, we consider $S_H=S'_H\cup \{0,-1 \}$, where we assume that $-1,0\notin S'_H$, and extend the well-order defined on $S'_H$ to $S_H$ as follows: $-1<0<\beta$ for every $\beta\in S'_H$.

For every $h\in H$ we take a well-ordered non-empty set $W_h$ such that $W_h$ is isomorphic to $W_t$ if and only if $h=t$. For every $h\in H$ let $w_h\in W_h$ denote the first element or minimum of $W_h$. We consider 
$$X_H^*=(H\times S_H)\cup (\bigcup_{\substack{{(h,\beta)\in G\times S_H} \\ 0\leq \beta <max( S_H)}} (S_{(h,\beta)}\cup T_{(h,\beta)})\cup (\bigcup_{h\in H} W_h))\cup \{ *\},$$
where 
$$S_{(h,\beta)}=\{A_{(h,\beta)},B_{(h,\beta)},C_{(h,\beta)},D_{(h,\beta)} \}, T_{(h,\beta)}=\{E_{(h,\beta)},F_{(h,\beta)},G_{(h,\beta)},H_{(h,\beta)},I_{(h,\beta)},J_{(h,\beta)} \},$$
and we have the following relations:
\begin{enumerate}
\item $(h,\beta)<(h,\gamma)$ if $-1\leq \alpha<\gamma$, where $h\in H$ and $\beta,\gamma\in S_H$.
\item $(h\beta,-1)\prec (h,\beta)$, where $h\in H$ and $\beta\in S_H\setminus \{-1,0\}$.
\item $A_{(h,\beta)}\succ C_{(h,\beta)},D_{(h,\beta)}$; $B_{(h,\beta)}\succ (h,\beta), C_{(h,\beta)}$ and $(h,\beta)\succ D_{(h,\beta)}$, where $h\in H$ and $\beta\in S_H\setminus \{-1\}$.
\item $E_{(h,\beta)}\succ {(h,\beta)},I_{(h,\beta)}$; $F_{(h,\beta)}\succ H_{(h,\beta)},J_{(h,\beta)}$; $G_{(h,\beta)}\succ I_{(h,\beta)},J_{(h,\beta)}$ and $(h,\beta)\succ H_{(h,\beta)}$, where $h\in H$ and $\beta\in S_H\setminus \{-1\}$.
\item $*\succ (h,-1)$, where $h\in H$.
\item We extend the partial order defined on $W_h$ to $X_H$ declaring that $A_{(h,0)}\prec w_h$, where $h\in H$.
\end{enumerate}
The remaining relations can be deduced from the above using transitivity. It is routine to verify that $X_H^*$ with the previous relations is a partially ordered set.

\begin{figure}[h]\center
\includegraphics[scale=1]{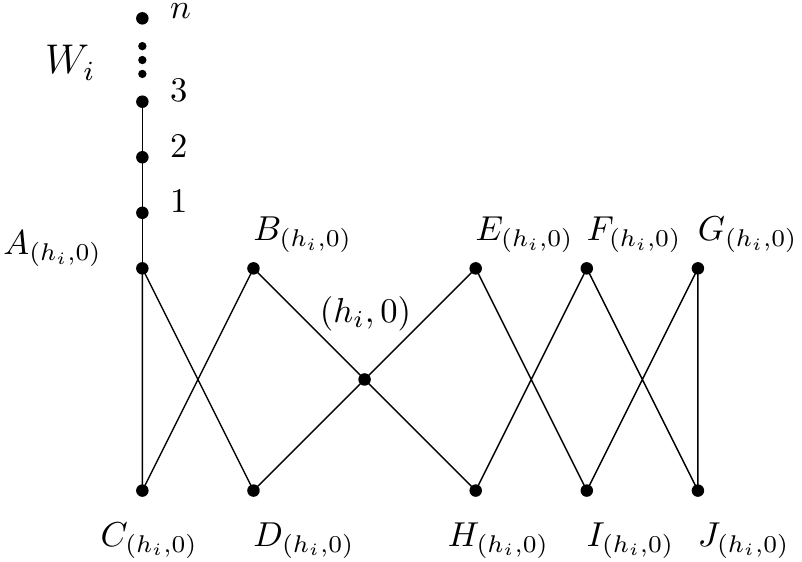}
\caption{Hasse diagram of $S_{(h,0)}\cup T_{(h,0)}\cup W_h$, where $W_h$ is a finite well-ordered set.}\label{fig:AspasNoFinitas}

\end{figure}

We proceed to show that $Aut(X_H^*)$ is the trivial group and $\mathcal{E}(X_H^*)\simeq H$. It is clear that $X_H^*$ and $X_H^*\setminus \{ W_h|h\in H\}$ have the same homotopy type. We define $r:X_H^*\rightarrow X_H^*\setminus \{ W_h|h\in H\}$ given by 
\[ r(x) = \begin{cases} 
          A_{(h,0)} & x\in W_h \\
          x & x\in X_H^*\setminus \{ W_h|h\in H\}.
       \end{cases}
    \]
It is trivial to show that $r$ is continuous and satisfies that $r(x)\leq id(x)$ for every $x\in X_H^*$, where $id:X_H^*\rightarrow X_H^*$ denotes the identity map. This implies that $X_H^*\setminus \{ W_h|h\in H\}$ is a strong deformation retract of $X_H^*$. On the other hand, $X_H^*\setminus \{ W_h|h\in H\}$ is the same topological space considered in \cite[Section 3]  {chocano2020topological} and denoted by $\overline{X}^*_H$. Therefore we know that $\mathcal{E}(X_H^*)\simeq \mathcal{E}(\overline{X}_H^*)\simeq  Aut(\overline{X}_H^*)\simeq H$, where $\phi:H\rightarrow Aut(\overline{X}_H^*)$ given by $\phi(t)(h,\beta)=(th,\beta)$ and $\phi(t)(S_{(h,\beta)}\cup T_{(h,\beta)})= S_{(th,\beta)}\cup T_{(th,\beta)}$ is an isomorphism of groups. 

The task is now to prove that $Aut(X_H^*)$ is the trivial group. Let us take $f\in Aut(X_H^*)$. We consider $A_{(h,0)}$ for some $h\in H$. Since $F_{A_{(h,0)}}\setminus \{A_{(h,0)}\}$ has a minimum $w_h$, it follows that $A_{(h,0)}$ is an up beat point. By Lemma \ref{lem:beatPoints}, we know that $f(A_{(h_i,0)})$ is also an up beat point. Therefore, $f(A_{(h_i,0)})$ is of the form $A_{(t,0)}$ for some $t\in H$. By Proposition \ref{prop:sucesores} we get that $f(w_h)=w_t$. It follows from the continuity of $f$ that $f(W_h)\subseteq W_t$. Since $f$ is a homeomorphism, we have that $f_{|W_h}$ is also a homeomorphism. Therefore we get that $h=t$; otherwise we would get a contradiction since $W_h$ is homeomorphic to $W_t$ if and only if $h=t$. Using Proposition \ref{prop:sucesores} it is easy to verify that $f$ fixes $S_{(h,0)}$ for every $h\in H$. On the other hand, \cite[Remark 4.2]{chocano2020topological} says that if a homeomorphism $g:X_H^*\setminus \{ W_h|h\in H\}\rightarrow X_H^*\setminus \{ W_h|h\in H\}$ coincides at one point with the identity map, then $g$ is the identity map. Thus, we can deduce that $f$ is the identity map and $Aut(X_H^*)$ is the trivial group.

\underline{\textit{Construction of $X_H^G$}}. We consider $X_H^G=X_H^*\cup X_*^G$, where we are identifying the point $*$ of both topological spaces, i.e., the partial order of $X_H^G$ preserves the relations defined on $X_H^*$ and $X^G_*$:
\begin{itemize}
\item If $x\in X_H^*$ and $y\in X_*^G$, then $x$ is smaller than $y$ if and only if $x\leq *$ and $*\leq y$.
\item If $x,y\in X_H^*$, then $x$ is smaller (greater) than $y$ if and only if $x$ is smaller (greater) than $y$ with the partial order defined on $X_H^*$.
\item If $x,y\in X_*^G$, then $x$ is smaller (greater) than $y$ if and only if $x$ is smaller (greater) than $y$ with the partial order defined on $X_*^G$.
\end{itemize}
It is evident that $\mathcal{E}(X_H^G)$ is isomorphic to $ H$ because $X_*^G$ is contractible to $*$ and $X_H^*$ is homotopy equivalent to $\overline{X}_H^*$. It suffices to show that $Aut(X_H^G)$ is isomorphic to $G$. We verify that every $f\in Aut(X_H^G)$ satisfies that $f(x)\in X_H^*$ for every $x\in X_H^*$ and $f(x)\in X_*^G$ for every $x\in X_*^G$. Firstly, we show that $*$ is a fixed point for every homeomorphism $f$. We have $ht(*)=1$. Since for every $x\in X_*^G\setminus \{* \}$ the height of $x$ is at least $2$ or different from $1$, it follows that $f(*)\notin X_*^G\setminus \{* \}\subset X_H^G$. The only elements of $X_H^*$ that have height one are of the form $*$ or $(h,0)$ or $A_{(s,\alpha)}$, $F_{(s,\alpha)}$, $E_{(s,\alpha)}$ for some $h,s\in H$ and $\alpha\in S_H\setminus \{-1\}$. We can discard the maximal elements, otherwise, $f^{-1}$ would send a maximal element to a non-maximal element. If $f(*)=A_{(h,0)}$, then we get a contradiction since $A_{(h,0)}$ is an up beat point. If $f(*)=(h,0)$ for some $h\in H$, then we get that $f^{-1}(E_{(h,0)})\succ f^{-1}(h,0)=*$ by Proposition \ref{prop:sucesores}. By Lemma \ref{lem:beatPoints} we have that $f^{-1}(E_{(h,0)})\neq (g,-1)$ for every $g\in G$ because $E_{(h,0)}$ is not a down beat point. Hence, the only possibility is $f(*)=*$. Finally, by the continuity of $f$, we get that $f(x)\in X_*^G\subset X_H^G$ for every $x\in X_*^G\subset X_H^G$. This implies that $Aut(X_H^G)$ is isomorphic to $G$.

We prove the remaining case. If $G$ is the trivial group, then it suffices to consider $X_H^*$ to conclude. If $H$ is the trivial group, then $X_*^G$ satisfies the desired properties.
\end{proof}
\begin{rem}\label{rem:coinciden} If $f,k\in Aut(X_H^G)$ are such that there exists $x\in X_*^G\setminus \{ *\}$ satisfying $f(x)=k(x)$, then $f=k$. This is a consequence of the isomorphism of groups $\varphi$ given in the proof of Lemma \ref{thm:teoremaPrincipal}. Similarly, if $[f],[k]\in \mathcal{E}(X_H^G)$ are such that there exists $x\in X_H^G\setminus (\{X_*^G \}\cup \{ W_h|h\in H\})$ satisfying that $f(x)=k(x)$, where $f\in [f]$ and $k\in [k]$, then $f=k$ and $[f]=[k]$. This is a consequence of the construction of $X_H^G$ and $\phi$ given in the proof of Lemma \ref{thm:teoremaPrincipal}.
\end{rem}
\begin{prop}\label{prop:WeakHomotopyTypeArbitraryGroups} Let $G$ and $H$ be groups. Then the Alexandroff space $X_H^G$ obtained in the proof of Lemma \ref{thm:teoremaPrincipal} has the weak homotopy type of the wedge sum of $3|H||S_H|$ circles when $H$ is a finite group and the wedge sum of $|\mathbb{N}|$ circles when $H$ is a non-finite countable set.
\end{prop}
\begin{proof}
We have that $X_H^G$ has the same homotopy type of $X_H^*$. Repeating the same arguments used in \cite[Proposition 6.1]{chocano2020topological} the desired result follows.
\end{proof}

\begin{rem} Let $G$ and $H$ be finite groups and let $X_H^G$ be the finite topological space obtained in the proof of Lemma \ref{thm:teoremaPrincipal}. We can remove $\{T_{(h,\beta)}| h\in H, \beta\in S_H\setminus \{-1,0\}\}$ from $X_H^G$. The resulting poset also satisfies that its group of homeomorphisms is isomorphic to $G$ and its group of homotopy classes of self-homotopy equivalences is isomorphic to $H$. This finite topological space has $|G|(|S_G|+2)+|H|(|S_H|+2)+4|S_H||H|+\frac{|H|(|H|+1)}{2}+1$ points. The first term corresponds to $X_*^G\setminus \{ *\}$, the second term corresponds to $X_H^*\setminus \{*\}$, the third term corresponds to the sets $S_{(h,\beta)}$, the fourth term corresponds to the points of the sets $W_h$ and the last term corresponds to the point $*$.
\end{rem}

We can change the sets $T_{(h,\beta)}$ from the proof of Lemma \ref{thm:teoremaPrincipal} by $T_{(h,\beta)}^n$ as in \cite[Section 5]{chocano2020topological}, where $T_{(h,\beta)}^n:=\{ x_1,x_2,...x_{n+3}, y_1,y_2,...,y_{n+3}\}$ with the following relations:
\begin{align}
(h,\beta)<x_1>y_2<x_3>y_4<...<x_{n+2}>y_{n+3}<x_{n+3}>y_{n+2}<x_{n+1}<...<x_2>y_1<(h,\beta), \\
(h,\beta)<x_1>y_2<x_3>y_4<...>y_{n+2}<x_{n+3}>y_{n+3}<x_{n+2}>y_{n+1}<...<x_2>y_1<(h,\beta).
\end{align}
We consider (1) for $n$ odd and (2) for $n$ even. We denote this poset by $X_{Hn}^{G}$. 
\begin{cor}\label{cor:infinitosEspaciosFinitos} Given two finite groups $G$ and $H$, there are infinitely many (non-homotopy-equivalent) topological spaces $\{ X_{Hn}^{G}\}_{n\in \mathbb{N}}$ such that $Aut(X_{Hn}^G)$ is isomorphic to $ G$ and $\mathcal{E}(X_{Hn}^G)$ is isomorphic to $H$ for every $n\in \mathbb{N}$.
\end{cor}

\begin{proof}
The proof is analogous to the proof of Lemma \ref{thm:teoremaPrincipal}. By Theorem \ref{thm:stonghomeoeshomoeq}, we have that the topological spaces are not homotopy equivalent due to their different cardinality after removing all the beat points one by one.
\end{proof}

%\begin{rem} If $N$ is a normal subgroup of $G$, $G/N$ denotes the quotient group. We have a natural homomorphism of groups  $p:G\rightarrow G/N$. We consider $X_{G/N}^G$, i.e., $Aut(X_{G/N}^G)$ is isomorphic to $G$ and $\mathcal{E}(X_{G/N}^G)$ is isomorphic to ${G/N}$. There is a natural way to send an element $f\in Aut(X_{G/N}^G)$ to $ \mathcal{E}(X_{G/N}^G)$, we only need to consider the homotopy class $[f]$. On the other hand, it is easy to deduce by the construction of $X_{G/N}^G$ that $[f]=[id]$ for every $f\in Aut(X_{G/N}^G)$ because every homeomorphism let $X_{G/N}^{*}$ fixed.
%\end{rem}
\section{Examples, remarks and proof of Theorem \ref{thm:realizacionHomo}}\label{section:pruebaSegundoTeorema}
The idea of this section is to modify the topological space obtained in Lemma \ref{thm:teoremaPrincipal} to prove Theorem \ref{thm:realizacionHomo}. Given a homomorphism of groups $f:G\rightarrow H$, we slightly modify the topological space $X_H^*$ defined in the proof of Lemma \ref{thm:teoremaPrincipal} to get $X_f$. Adding new relations to $X_H^*$ we can control the homomorphism of groups $\tau:Aut(X_f)\rightarrow \mathcal{E}(X_f)$ given by $\tau(f)=[f]$. 
\begin{ex} Let us consider the cyclic group of two elements $\mathbb{Z}_2$ and the group of integer numbers $\mathbb{Z}$. We consider the homomorphism of groups $f:\mathbb{Z}\rightarrow \mathbb{Z}_2$ given by $f(n)=n$ $mod$ $2$. We consider the topological space $X_{\mathbb{Z}_2}^\mathbb{Z}$ obtained in the proof of Lemma \ref{thm:teoremaPrincipal}. We remove $W_0$ and $W_1$ from it. The resulting poset $X_f$ corresponds to the Hasse diagram shown in black in Figure \ref{fig:realHomoEjemplo}. 
\begin{figure}[h]\center
\includegraphics[scale=0.82]{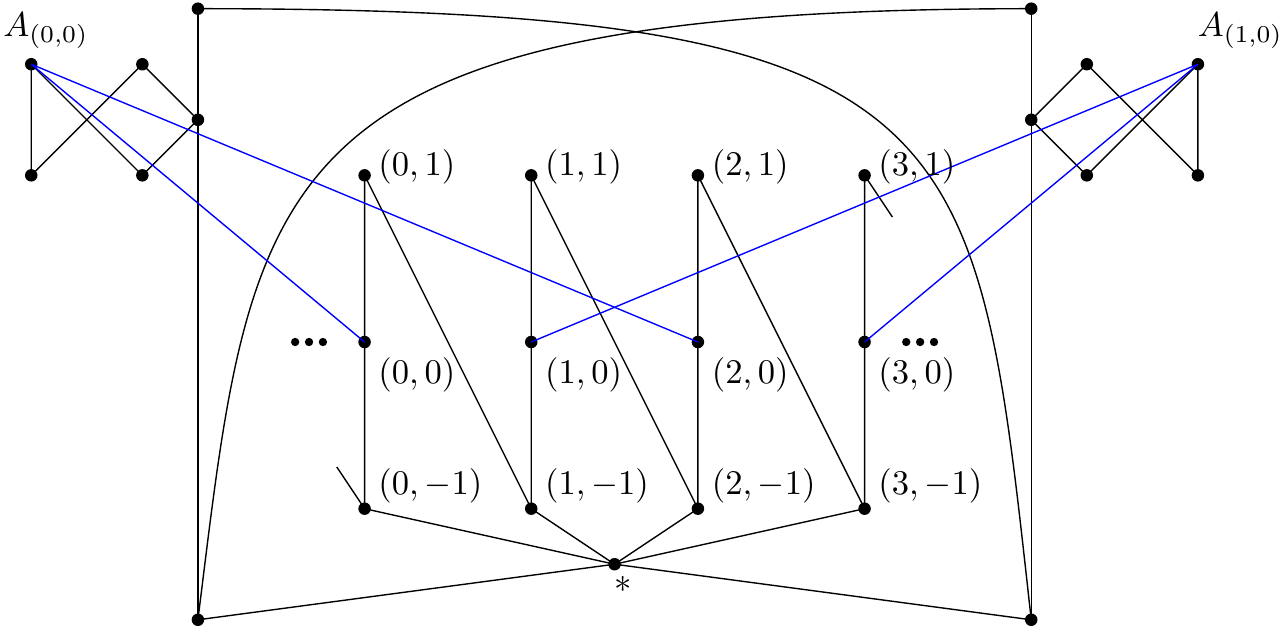}
\caption{Hasse diagram of $X_f$.}
\label{fig:realHomoEjemplo}
\end{figure}

Now we add the following relations to $X_f$: $(n,0) \prec A_{(1,0)}$ if $f(n)=1$ and $(n,0)\prec A_{(0,0)} $ if $f(n)=0$. In Figure \ref{fig:realHomoEjemplo} we have represented these relations in blue. It is easy to verify that $X_f$ satisfies that $Aut(X_f)=\mathbb{Z}$, $\mathcal{E}(X_f)= \mathbb{Z}_2$ and $f=\tau$.

\end{ex}

\begin{ex} Let $f:\mathbb{Z}_2\rightarrow  \mathbb{Z}_2\oplus \mathbb{Z}_2$ be the homomorphism of groups given by $f(0)=(0,0)$ and $f(1)=(1,0)$. In Figure \ref{fig:realHomoEjemploKlein} we have the Hasse diagram of $X_f$, where we use the same notation introduced in Example \ref{ex:primerEjemplo}. We have that $Aut(X_f)= \mathbb{Z}_2$, $\mathcal{E}(X_f)=\mathbb{Z}_2\oplus \mathbb{Z}_2$ and $\tau=f$.
\begin{figure}[h]\center
\includegraphics[scale=0.82]{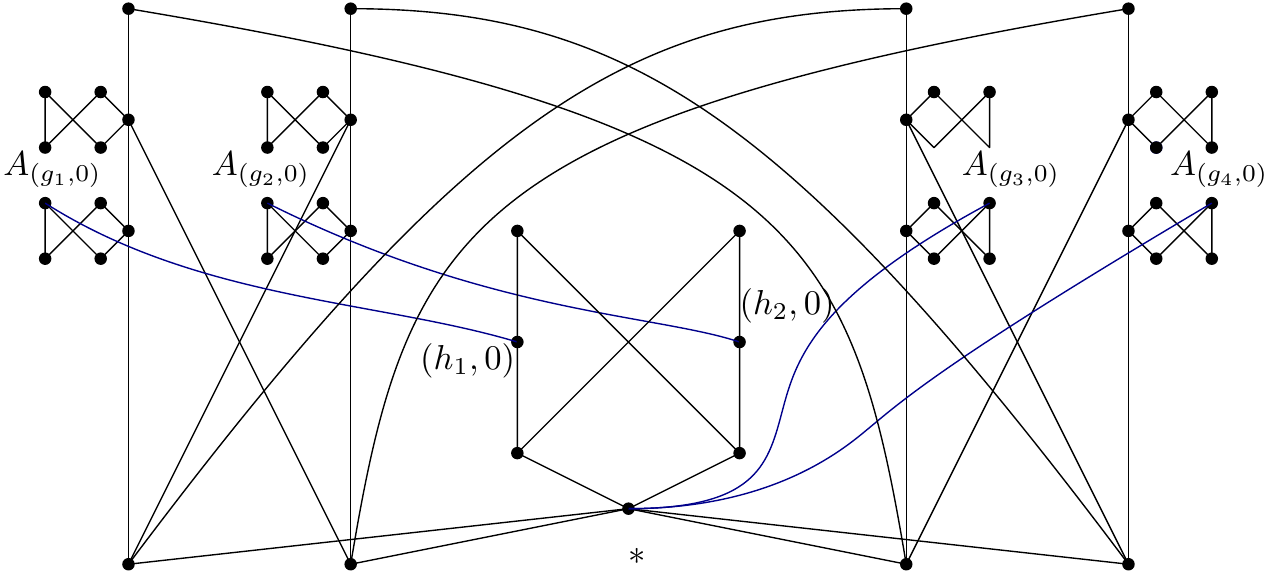}
\caption{Hasse diagram of $X_f$.}
\label{fig:realHomoEjemploKlein}
\end{figure}
\end{ex}

\begin{proof}[\textbf{Proof of Theorem \ref{thm:realizacionHomo}}]
Suppose $G$ and $H$ are not trivial groups since otherwise the result follows from Lemma \ref{thm:teoremaPrincipal}. Let $X_H^G$ be the topological space obtained in the proof of Lemma \ref{thm:teoremaPrincipal}. We consider $X_f=X_H^G\setminus \{W_h |h\in H\}$. We keep the same relations defined on $X_f$ as a subspace of $X_H^G$ and add the following relations:
\begin{itemize}
\item $A_{(h,0)}\succ (g,0)$ if $f(g)=h$, where $h\in H$ and $g\in G$.
\item $A_{(h,0)}\succ*$ if $h\notin f(G)$.
\end{itemize}
It is easy to check that $X_f$ is a partially ordered set with the above relations. The task is now to show that $\mathcal{E}(X_f)\simeq H$. We consider $r:X_f\rightarrow X_f$ given by 
\[ r(x) = \begin{cases} 
          * & x\in X_*^G \\
          x & x\in X_f  \setminus \{ X_*^G\}.
       \end{cases}
    \]
We have that $r$ preserves the order so it is a continuous map. It is simple to verify that $r(x)\leq id(x)$ for every $x\in X_f$, where $id$ denotes the identity map. From this it follows that $X_f$ is homotopy equivalent to $r(X_f)=X_H^*\subset X_f$. On the other hand, repeating the same arguments used in \cite{chocano2020topological}, it can be proved that $X_H^*$ is locally a core \cite{kukiela2010homotopy} or a minimal finite space in case $H$ is a finite group, which implies that $\mathcal{E}(X_f)\simeq \mathcal{E}(X_H^*)= Aut(X_H^*)$. Since $P_*=(|H|,|H|)$ and $ht(*)=1$, it follows that every homeomorphism $T:X_H^*\rightarrow X_H^*$ fixes $*$. Hence, the group of homeomorphisms of $X_H^*$ as a subspace of $X_f$ is isomorphic to the group of homeomorphisms of the topological space $X_H^*$ obtained in the proof of Lemma \ref{thm:teoremaPrincipal}. Thus, $\mathcal{E}(X_f)\simeq H$.

We proceed to show that $Aut(X_f)\simeq G$. We consider the following auxiliary sets: $Col_h=\{(h,\beta)|\beta\in S_H\}\cup \{ S_{(h,\beta)} \cup T_{(h,\beta)}| \beta\in S_H\setminus \{-1,max(S_H) \} \}$, where $h\in H$, and $Col^g=\{(g,\alpha)|\alpha\in S_G \}$, where $g\in G$. If $x\in X_*^G\subset X_f$, then every homeomorphism $T:X_f\rightarrow X_f$ satisfies that $T(x)\in X_*^G$. We prove the last assertion. We know that $X_f\setminus \{X_* ^G \}$ does not contain beat points. On the other hand,  for every $g\in G$ we have that $(g,-1)$ is a beat point of height $2$. Using Proposition \ref{prop:sucesores} and the notion of continuity we deduce that for every $T\in Aut(X_f)$ and $(g,\alpha)$, where $g\in G$ and $\alpha \in S_G$, we have $T(g,\alpha)=(g',\alpha)$ for some $g'\in G$. 

We consider $\varphi:G\rightarrow Aut(X_f)$ given by $\varphi(g)(g',\alpha)=(gg',\alpha)$ if $g'\in G$ and $\alpha\in S_G$, $\varphi(g)(h,\beta)=(f(g)h,\beta)$ if $h\in H$ and $\beta \in S_H$, $\varphi(g)(S_{(h,\beta )}\cup T_{(h,\beta)} )=S_{(f(g)h,\beta )}\cup T_{(f(g)h,\beta)} $ defined in the natural way if $h\in H$ and $\beta \in S_H\setminus \{-1,max(S_H) \}$ and $\varphi(g)(*)=*$. We prove that $\varphi$ is well-defined. We verify the continuity of $\varphi(g)$. Suppose $(g',0)\prec A_{(h,0)}$ for some $g'\in G$ and $h\in H$. By hypothesis, $f(g')=h$. Therefore,
 $$\varphi(g)(g',0)=(gg',0)\prec A_{(f(gg'),0)} = A_{(f(g)f(g'),0)}=A_{(f(g)h,0)}=\varphi(g) A_{(h,0)}.$$
It is easy to check that $\varphi(g)$ preserves the remaining relations. The inverse of $\varphi(g)$ is given by $\varphi(g^{-1})$. Hence, $\varphi$ is well-defined. By construction, $\varphi$ is a monomorphism of groups. Suppose $T\in Aut(X_f)$. Proposition \ref{prop:sucesores}, Remark \ref{rem:coinciden} and the fact that $T_{|X_*^G}\in Aut({X_*^G})$ imply that every $T\in Aut(X_f)$ satisfies $T(Col_h)=Col_{h'}$ and $T(Col^g)=Col_{g'}$ for some $g'\in G$ and $h'\in H$, where $g\in G$ and $h\in H$. We consider $(g,0)\prec A_{(h,0)}$ for some $g\in G$ and $h\in H$. We get $T(Col^g)=Col^{g'}$ for some $g'\in G$ and $T(Col_h)=Col_{h'}$ for some $h'\in H$. By Remark \ref{rem:coinciden}, the proof of Lemma \ref{thm:teoremaPrincipal} and the fact that $T_{|X_G^*}\in Aut(X_*^G)$, there exists $t\in G$ such that  $T(Col^s)=Col^{ts}$, where $s\in G$. Hence, $g'=tg$. By Proposition \ref{prop:sucesores}, $T(A_{(h,0)})=A_{(h',0)}\succ (tg,0)=T(g,0)$, we have $h'=f(tg)=f(t)f(g)=f(t)h$. Thus, $T=\varphi(t) $ because of Remark \ref{rem:coinciden} and the fact that $T_{|X_f\setminus \{X_*^G \} }\in Aut(X_H^*\setminus \{ *\})$. By construction, for every $g\in G$ the equality $\tau(g)=f(g)$ holds. Since every $T\in Aut(X_f)$ can be seen as $T=\varphi(g)$ for some $g$, it follows that $\tau(T)=f(g)$, where $f(g)=\varphi(g)_{|X_H^*}\in \mathcal{E}(X_f)$.
\end{proof}
\begin{rem} Theorem \ref{thm:realizacionHomo} generalizes Lemma \ref{thm:teoremaPrincipal} and the results of realization obtained in \cite{chocano2020topological}. Let $G$ and $H$ be two groups. Using Theorem \ref{thm:realizacionHomo}, we obtain a family of topological spaces $\{X_f\}_{f:G\rightarrow H}$ satisfying that $Aut(X_f)\simeq G$ and $\mathcal{E}(X_f)\simeq H$.
\end{rem}
\begin{prop} Let $G$ and $H$ be groups. If $g,f:G\rightarrow H$ are homomorphisms of groups, then $X_f$ is homotopy equivalent to $X_g$.
\end{prop}
\begin{proof}
The result is an immediate consequence of the construction. We have that $X_f$ is homotopy equivalent to $X_H^*$ for every homomorphism of groups $f:G\rightarrow H$. Therefore, the homotopy type of the topological space obtained in the proof of Theorem \ref{thm:realizacionHomo} does not depend on the homomorphism chosen to construct it. Thus we deduce the desired result.
\end{proof}
\begin{prop} Let $G$ and $H$ be groups and let $f,g:G\rightarrow H$ be homomorphisms of groups. Then $f=g$ if and only if $X_f$ is homeomorphic to $X_g$.
\end{prop}
\begin{proof}
One of the implications is trivial. It suffices to show that if $X_f$ is homeomorphic to $X_g$, then $f=g$. Since $X_f$ is homeomorphic to $X_g$, it follows that there exists a homeomorphism $T':X_f\rightarrow X_g$. From the construction of $X_f$ and $X_g$ in the proof of Theorem \ref{thm:realizacionHomo} it can be easily deduced that $T'_{|X^G_*}\in Aut(X^G_*)\simeq G$ and $T'_{|X_H^*}\in Aut(X_H^*)\simeq H$. This is due to the fact that $X^G_*$ contains beat points while $X_H^*$ does not have beat points. Therefore, $T'_{|X^G_*}$ can be related to the action of an element $T\in G$ and $T'_{|X_H^*}$ can be related to the action of an element $\overline{T}\in H$. We have $(e,1)\prec A_{(f(e),0)}$, where $e$ denotes the identity element in $G$, and we also have 
$$T'(e,1)=(T,1)\prec A_{(\overline{T}f(e),0)}=T'( A_{(f(e),0)}), $$
which implies that $g(T)=\overline{T}f(e)$. Thus, $g(T)=\overline{T}$ because $f$ is a homomorphism of groups. In addition, for every $h\in G$,  we know that there exists a relation in $X_f$ of the following form $(h,1)\prec A_{(f(h),0)}$. We have
$$T'(h,1)=(Th,1)\prec A_{(\overline{T}f(h),0)}=T'(A_{(f(h),0)}).$$
By the construction of $X_g$ we get $g(Th)=g(T)g(h)=\overline{T}f(h)$. Earlier we prove that $g(T)=\overline{T}$, which implies that $g(h)=f(h)$ for every $h\in G$.
\end{proof}
\section{Groups of homology, homotopy and automorphisms}\label{section:homologyHomotopy}

We first prove that the groups studied previously do not determine neither the homotopy type nor the topological type of a topological space $X$ in general. To do this we provide an example. However, if the topological space $X$ satisfies some properties, namely, $X$ is compact and a locally Euclidean manifold with or without boundary, then its group of homeomorphisms determines its topological type of it, see \cite{whittaker1963onIsomorphic} for more details.
\begin{ex}\label{ex:ejemploTrivial} Let us consider the Alexandroff space $W_2$ given by the Hasse diagram of Figure \ref{fig:ejemploTrivial}. It is the union of $L_1=\{x_i \}_{i=1,...,9}$ and $L_2=\{x_j\}_{j=9,...17}$, where we are identifying the point $x_9$ of $L_i$ for $i=1,2$. For simplicity, $W_1$ denotes $L_1$. The topological space $L_1$ was introduced in  \cite[Figure 2]{rival1976afixed} and has the weak homotopy type of a point. It is proved in \cite{cianci2020smallest} that $W_1$ is the smallest finite topological space having the same weak homotopy type of a point but not contractible. It is clear that $W_2$ does not have beat points, so $Aut(W_2)$ is isomorphic to $\mathcal{E}(W_2)$. On the other hand, it is easy to show that $Aut(W_2)$ is the trivial group. Since $P_x=(3,5)$ only for $x\in \{x_5,x_{13}\}$ and the heights of these points are different, it follows that every homeomorphism must fix $x_5$ and $x_{13}$. Proposition \ref{prop:sucesores} leads to the desired result. Furthermore, the homotopy and singular homology groups of $W_2$ are trivial. We can study the weak homotopy type of $W_2$ studying the McCord complex $\mathcal{K}(W_2) $ or removing beat and weak beat points. We have that $x_{16}$ is a weak beat point. After removing this point, $x_{12}$ and $x_{14}$ are up beat points. If we remove them, then it is easy to check that the remaining space is homotopy equivalent to the space given by the points $\{x_i\}_{i=1,...,9}$. We continue in this fashion. We have $x_{8}$ is a weak beat point. After removing this point, $x_{7}$ and $x_{9}$ are down beat points. Thus, $X$ has the same weak homotopy type of a point. Therefore we have that $W_2$ is a topological space satisfying that $Aut(W_2)\simeq \mathcal{E}(W_2)\simeq \pi_n(W_2)\simeq H_n(W_2)\simeq 0$ for every $n>0$ but it is not homeomorphic nor homotopy equivalent to a point. We can generalize this topological space by taking more copies of the topological space introduced in \cite{rival1976afixed}. For instance, we can define $W_3$ just as $L_3 \cup W_2$, where $L_3=\{x_i\}_{i=17,...,25}$ and we are identifying the point $x_{17}$ of $W_3$ and $W_2$. It is easy to prove that $W_n$ has the weak homotopy type of a point for every $n\in \mathbb{N}$ because $x_i\in W_n$ with $i\equiv 0\ (\textrm{mod}\ 8)$ is a weak beat point.
\begin{figure}[h]\center
\includegraphics[scale=1.3]{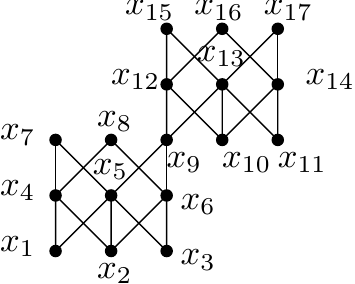}
\caption{Hasse diagram of $W_2$.}
\label{fig:ejemploTrivial}
\end{figure}
\end{ex}
One possible consequence of the previous construction is the following result.
\begin{prop}\label{prop:circfuera} Let $G$ and $H$ be finite groups. There exists a topological space $X$ such that $Aut(X)$ is isomorphic to $G$, $\mathcal{E}(X)$ is isomorphic to $H$ and $X$ is weak homotopy equivalent to a point.
\begin{proof}
We consider the topological space $X_H^G$ obtained in the proof of Lemma \ref{thm:teoremaPrincipal} and the finite topological space $W_2$ given in Example $\ref{ex:ejemploTrivial}$. Let $X$ denote $X_H^G\circledast W_2$. By Proposition \ref{non-hausdorffsuspensionAut}, Example \ref{ex:ejemploTrivial} and the proof of Lemma \ref{thm:teoremaPrincipal} we have that $Aut(X)\simeq Aut(X_H^G)\times Aut(W_2)\simeq Aut(X_H^G)\simeq G$. We get that $X$ is homotopy equivalent to $X_H^*\circledast W_2$ by removing its beat points one by one. Since $X_H^*\circledast W_2$ does not contain beat points by Corollary \ref{cor:autesigualaE}, it follows that $\mathcal{E}(X_H^*\circledast W_2)\simeq Aut(X_H^*\circledast W_2)\simeq Aut(X_H^*)\simeq H$. In addition, since $W_2$ is collapsible, it follows that $X$ has the weak homotopy type of a point by Remark \ref{rem:nonhausdorffsuspensioncontractible}
\end{proof}
\end{prop}
%\begin{rem} A similar statement can be obtained for non-finite groups applying the same techniques used in \cite{chocano2020topological}. The construction is the same of Proposition \ref{prop:circfuera}. It is only necessary to prove that the topological space obtained satisfies that it is a locally core. Then, $Aut(X)\simeq \mathcal{E}(X)$.
%\end{rem}
%\begin{rem} If we use the topological spaces constructed in Example \ref{ex:ejemploTrivial} instead of $S_{{(g,i)}}$ in the construction of $X_H^*$ when $H$ is finite, i.e., we identify $(g,i)$ with $x_6$, we can reduce the number of circles that appear in the wedge sum when we study the weak homotopy type of $X_H^G$, Proposition \ref{prop:WeakHomotopyTypeArbitraryGroups}. In this case, $X_H^G$ would have the same weak homotopy type of a wedge sum of $|S_H||H|$ circles. It is routine to check that $Aut(X_H^G)$ is isomorphic to $G$ and $\mathcal{E}(X_H^G)$ is isomorphic to $H$.
%\end{rem}
%The same result may also be obtained when $H$ is not a finite group applying the theory developed in \cite{kukiela2010homotopy}. 
We motivate the proof of Theorem \ref{thm:homologygroupsAutE} with one example.
\begin{ex}\label{ex:completo} Let us consider $G=\mathbb{Z}_3$ and $H=\mathbb{Z}_2$. We consider the minimal finite model of the 2-dimensional sphere $X$, that is, $X=\{A,B,C,D,E,F\}$, where $A,B>C,D,E,F$ and $C,D>E,F$. Then $|\mathcal{K}(X)|$ is homeomorphic to $S^2$. 

We want to find a finite $T_0$ topological space $\overline{X}_{H}^{G}$ such that $H_n(\overline{X}_{H}^{G})$ and $\pi_n(\overline{X}_{H}^{G})$ are isomorphic to $ H_n(S^2)$ and $\pi_n(S^2)$ respectively for every non-negative integer $n$, $Aut(\overline{X}_{H}^{G})$ is isomorphic to $\mathbb{Z}_3$ and $\mathcal{E}(\overline{X}_{H}^{G})$ is isomorphic to $\mathbb{Z}_2$. The idea is to modify $X$ to obtain a new space satisfying that its group of homeomorphisms is trivial. We enumerate the points of $X$. For each $i\in X$ with $i=1,...,|X|$, we add $W_i$ to $X$ as in Example \ref{ex:ejemploTrivial}. The Hasse diagram of the new topological space, denoted by $X'$, can be seen in Figure \ref{fig:S2finito}. The Hasse diagram of $X$ is painted black whereas the new part is blue and purple. In purple we have the weak beat points that are not beat points. It is clear that $X'$ does not have beat points so $Aut(X')$ is isomorphic to $ \mathcal{E}(X')$. A homeomorphism $f$ sends weak beat points to weak beat points by Lemma \ref{lem:beatPoints}. It is easy to deduce from this that $Aut(X')$ is the trivial group. On the other hand, the new structure added can be removed without changing the weak homotopy type of the space, see Example \ref{ex:ejemploTrivial}. Therefore, we have that $H_n(X')$ is isomorphic to $H_n(X)$ and $\pi_n(X')$ is isomorphic to $\pi_n(X)$ for every non-negative integer $n$. 

Finally, we add a new point $t$ that connects $X'$ to $X_{H}^{G}\circledast W_2$, where $X_{H}^{G}$ is the space obtained the proof of Lemma \ref{thm:teoremaPrincipal} for finite groups. In Figure \ref{fig:S2finito} we have the Hasse diagram of the new topological space $\overline{X}_{G}^{H}$. The relations with the point $t$ are shown in green. The Hasse diagram of $X_{H}^{G}\circledast W_2$ is painted read and orange. It is easy to check that $\overline{X}_{H}^{G}$ satisfies the desired properties.

\begin{figure}[h]\center
\includegraphics[scale=0.74]{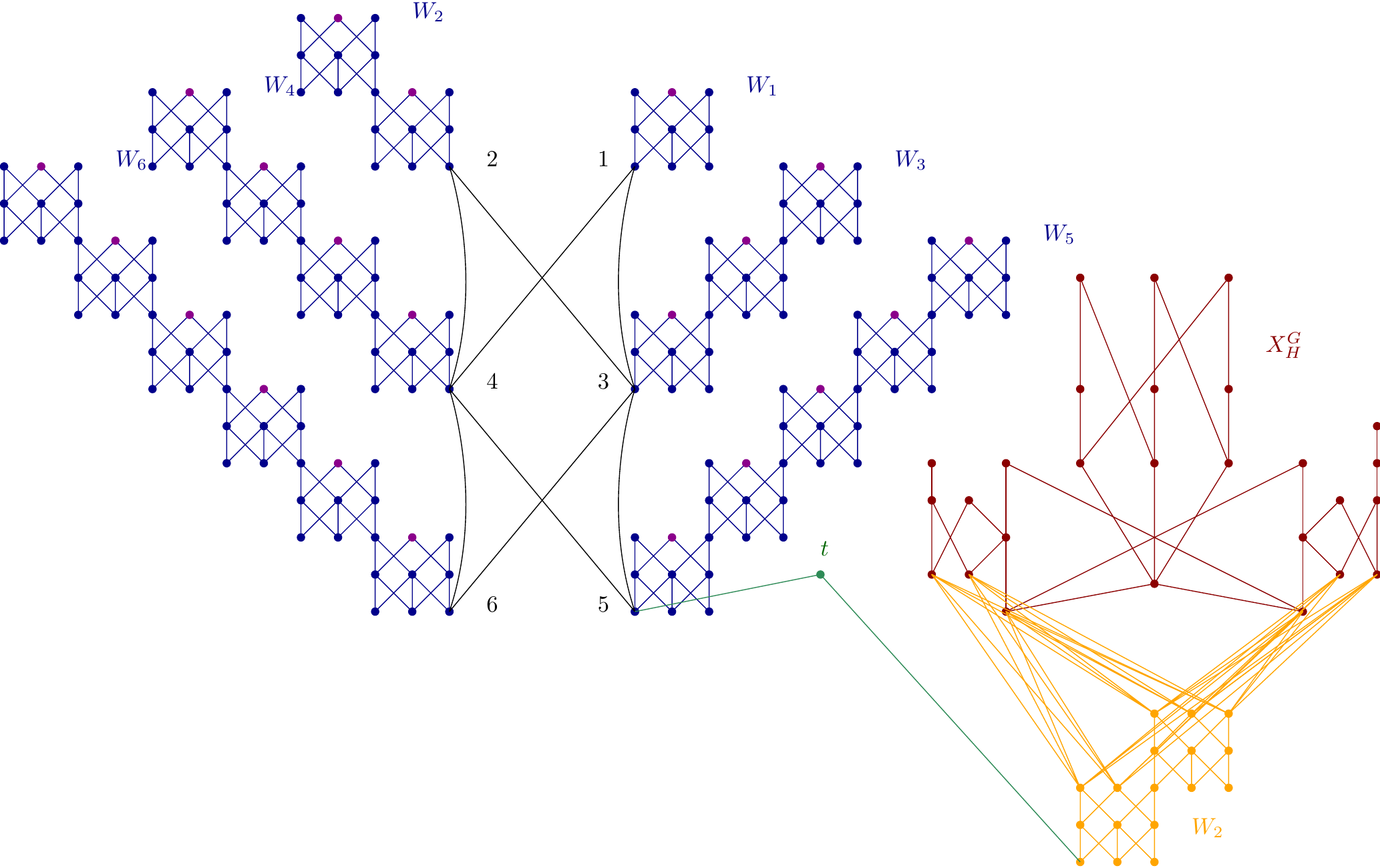}
\caption{Hasse diagram of $\overline{X}_{H}^{G}$.}\label{fig:S2finito}
\end{figure}

%\begin{figure}[h]\center
%\includegraphics[scale=0.5]{GrupoHomotopiaS1.pdf}
%\caption{Hasse diagram of $X$.}\label{fig:S1finito}
%\end{figure}
%

%It is easy to check that $X$ has the same homotopy type than $\mathcal{X}(S^1)$. $X_*^{\mathbb{Z}_2}$ is collapsible to $x_1$. The points of the form $(i,j)$ are all of them beat points so we can remove it without change the homotopy type. Finally, $Aut(X)=Aut(X_*^{\mathbb{Z}_2})=\mathbb{Z}_2$ because the beat points added make $X$ rigid in terms of homeomorphism in every part except in $X_*^{\mathbb{Z}_2}$.
\end{ex}
%
%
%Using the construction of the Eilenberg-MacLane spaces or the Moore spaces we can obtain directly the following results.
%\begin{cor}\label{cor:EilenbergMaclaneCorrecto}
%Let $G$ be any group, $H$ be an abelian finitely generated group and $n$ be a natural number. There exists a topological space $X$ such that $Aut(X)=G$ and $\pi_n(X)=H$.
%\end{cor}
%\begin{cor}\label{cor:EilenbergMaclaneCorrecto}
%Let $G$ be any group, $H$ be a finitely generated group and $n$ be a natural number. There exists a topological space $X$ such that $Aut(X)=G$ and $H_n(X)=H$.
%\end{cor}

\begin{proof}[\textbf{Proof of Theorem \ref{thm:homologygroupsAutE}}]
If $X$ has the same homotopy type of a point, then the result can be deduced from Proposition \ref{prop:circfuera}. Therefore we can assume that $X$ does not have the same homotopy type of a point. The idea of the proof is to follow techniques similar to the ones used in the proof of Lemma \ref{thm:teoremaPrincipal}. From the simplicial approximation to CW-complexes, \cite[Theorem 2C.5.]{hatcher2000algebraic} we get that there exists a finite simplicial complex that is homotopy equivalent to $X$. By abuse of notation we continue to write $X$ for the finite simplicial complex. We apply the McCord functor $\mathcal{X}$ to $X$ in order to obtain a finite $T_0$ topological space $\mathcal{X}(X)$ such that $\mathcal{X}(X)$ is weak homotopy equivalent to $X$. We can suppose that $\mathcal{X}(X)$ does not have beat points or weak beat points; otherwise we can remove them one by one until there are none left. We denote by $n=|\mathcal{X}(X)|$ and label the points in $\mathcal{X}(X)$, that is, $\mathcal{X}(X)=\{y_i \}_{i=1...n}$. For each $y_i\in \mathcal{X}(X)$ we consider $W_i$, where $W_i$ is the topological space obtained in Example \ref{ex:ejemploTrivial}. We consider $Z=\mathcal{X}(X) \cup \bigcup_{i=1,...,n} W_i $, where we are identifying the point $y_i$ with $x_1\in W_i$ for every $i=1,...,n$. We define the partial order on $Z$ extending the already existing partial orders. To do that we use transitivity, i.e., for $x,y\in Z$, $x\geq y$ if and only if one of the following situations is satisfied:
\begin{itemize}
\item $x,y\in \mathcal{X}(X)$ and $x$ is greater than $y$ with the partial order defined on $\mathcal{X}(X)$.
\item $x,y\in W_i$ for some $i$ and $x$ is greater than $y$ with the partial order defined on $W_i$.
\item $y\in \mathcal{X}(X)$, $x\in W_i$ for some $i$ and $x\geq y_i(=x_1) \geq  y$.
\end{itemize}

Consider $\overline{X}_H^G= Z\cup X_H^G\circledast W_2\cup \{t \}$, where $X_H^G$ is the space obtained in the proof of Lemma \ref{thm:teoremaPrincipal} and $W_2$ the space given in Example \ref{ex:ejemploTrivial}. We extend the partial order defined on $Z$ and $X_H^G\circledast W_2$ to $\overline{X}_H^G$ by declaring that $y_j<t>x_1$ for some $h\in H$, where $y_j$ is a minimal point in $\mathcal{X}(X)$ and $x_1\in W_2$. We prove that $f\in Aut(\overline{X}_H^G)$ restricted to $Z$ is the identity. In $W_i$ there are $i$ weak beat points that we will denote by $z_j^i$ with $j=1,...,i$. In fact, we have that the only weak beat points that are not beat points or do not have a bigger beat point are in $Z$. Hence, if $z_j^i\in W_i$ is a weak beat point, we have $f(z_j^i)=z_l^k\in W_k$ for some $l\leq k\leq n $. By Proposition \ref{prop:sucesores} and Lemma \ref{lem:beatPoints}, we have that $f(W_i)=W_k$. But $W_i$ is homeomorphic to $W_k$ if and only if $i=k$. By the continuity of $f$, $f(W_i)=id(W_i)$, so $f_{|Z}=id(Z)$, as we wanted. It is easy to check that $t$ is also a fixed point for every homeomorphism since $y_j\prec t$ and $f(y_j)=y_j$. We get $Aut(\overline{X}_H^G)\simeq Aut(X_H^G\circledast W_2)\simeq Aut(X_H^G)\times Aut(W_2)$. By Proposition \ref{non-hausdorffsuspensionAut}, Example \ref{ex:ejemploTrivial} and the Proof of Lemma \ref{thm:teoremaPrincipal}, $Aut(\overline{X}_H^G)\simeq Aut(X_H^G)\simeq Aut(X_*^G)\simeq G$. On the other hand, $\mathcal{E}(\overline{X}_H^G)\simeq \mathcal{E}(Z \cup \{t \}\cup X_H^*\circledast W_2)$, but $Z \cup \{t \}\cup X_H^*\circledast W_2$ does not contain beat points. Therefore, by Corollary \ref{cor:autesigualaE},  $\mathcal{E}(Z \cup \{t \}\cup X_H^*\circledast W_2)\simeq Aut(Z \cup \{t \}\cup X_H^*\circledast W_2)$. From here, repeating similar arguments than the ones used before, it can be deduced that $Aut(Z \cup \{t \}\cup X_H^*\circledast W_2) \simeq Aut(X_H^*)\simeq H$. 

Finally, $|\mathcal{K}(\overline{X}_H^G)|$ is clearly the wedge sum of $|\mathcal{K}(Z)|$ and $|\mathcal{K}(X_H^G\circledast W_2)|$. From Remark \ref{rem:nonhausdorffsuspensioncontractible} we obtain that $\mathcal{K}(X_H^G\circledast W_2)$ is homotopy equivalent to a point since $W_2$ is collapsible, which implies that $\mathcal{K}(W_2) $ is also collapsible, and $X_H^G$ is homotopy equivalent to $X_H^*$. We also get that $|\mathcal{K}(Z)|$ is homotopy equivalent to $X$ because every $W_i$ can be removed following the steps of Example \ref{ex:ejemploTrivial} without changing the weak homotopy type of $Z$. Therefore, for every $n\in \mathbb{N}$, we have $\pi_n(X)\simeq \pi_n(Z)$ and $H_n(X)\simeq H_n(Z)$.
%\begin{align*}
%H_n(|\mathcal{K}(\overline{X}_H^G)|)\simeq H_n(|\mathcal{K}(Z)|\vee |\mathcal{K}(X_H^G\circledast Z_2)|)\simeq H_n(X)\oplus H_n(*)\simeq H_n(X).
%\end{align*}
\end{proof}

%\begin{lem}\label{cor:autypi_n} If $K$ is a compact CW-complex and $G$ is a group, then there exists an Alexandroff space $\overline{X}^G$ such that $Aut(\overline{X}^G)\simeq G$, $\pi_n(\overline{X}^G)\simeq \pi_n(K)$ and $H_n(\overline{X}^G)\simeq H_n(K)$ for every non-negative integer $n$.
%\begin{proof}
%We only need to repeat a similar construction than the one used  in the proof of Lemma \ref{thm:homologygroupsAutE}. We consider $X_*^G$ from the proof of Theorem \ref{thm:teoremaPrincipal}. Therefore, $X_*^G$ is contractible to $*$. Without loss of generality, we can assume that $\mathcal{X}(K)$ does not have weak beat points or beat points. For each $x\in \mathcal{X}(K)$, we take a well-ordered non-empty set $W_x$ such that $W_x$ is homeomorphic to $W_y$ if and only if $x=y$. We identify the minimum of $W_x$ with $x$ and extend the partial order defined on both posets using transitivity, that is to say, if $z\in \mathcal{X}(K)$ and $w\in W_x$, we have that $z<w$ if and only if $z\leq x\leq w$. It is easy to check that $X=\mathcal{X}(K)\cup (\bigcup_{x\in \mathcal{X}(K)}W_x)$ satisfies that $Aut(X)$ is the trivial group. We take $x'$, a minimal point in $\mathcal{X}(K)$, we also add the point $t$ to $X\cup X_*^G$ satisfying that $x'<t>*$. Then, $X^G=X\cup X_*^G\cup \{t\}$ is the desired space. 
%\end{proof}
%\end{lem}

\begin{proof}[\textbf{Proof of Corollary \ref{cor:sequenceOfGroups}}]
It is an immediate consequence of Theorem \ref{thm:homologygroupsAutE}. We only need to consider the wedge sum of Moore spaces and then apply Theorem \ref{thm:homologygroupsAutE}.
\end{proof}

\begin{proof}[\textbf{ Proof of Corollary \ref{cor:homotopyGroupsandAuotomorphis}}]
We only need to use the beginning of the construction of Eilenberg-Maclane spaces to obtain a compact CW-complex $X$ with $\pi_n(X)\simeq H$ and possibly non-trivial higher homotopy groups. Therefore, the result is an immediate consequence of Theorem \ref{thm:homologygroupsAutE}.
\end{proof}

\begin{rem} The results obtained in this section are stated in terms of finite groups but it may be possible to get the same results for general groups. The idea could be to use the same constructions described in this section and the theory of \cite{kukiela2010homotopy}, which is a generalization of the theory of R.E. Stong \cite{stong1966finite}.
\end{rem}

For a compact $CW$-complex $X$, $\mathcal{E}_*(X)$ and $\mathcal{E_\#}(X)$ are nilpotent groups, see for instance \cite[Section 4]{costoya2020primer}. $\mathcal{E}_\#(X)$ ($\mathcal{E}_*(X)$) can be seen as the kernel of a homomorphism of groups. We consider the functor $\pi$ $(H_*)$ between $HPol$ and the category of groups given by $\pi(X)=\bigoplus_{i=1}^{dim(X)}\pi_i(X)$ ($H_*(X)=\bigoplus_{i=1}^{\infty} H_i(X)$). It is easy to check that $\pi$ $(H_*)$ induces a homomorphism of groups, $\overline{\pi}:\mathcal{E}(X)\rightarrow Aut(\pi(X))$ $(\overline{H}_*:\mathcal{E}(X)\rightarrow Aut(H_*(X)))$. This sends each self-homotopy equivalence to its induced morphism in the homotopy groups (homology groups), where $Aut(\cdot)$ denotes here the group of automorphisms of a group in the category of groups. Then, $\mathcal{E}_\#(X)$ ($\mathcal{E}_*(X)$) can be seen as the kernel of $\overline{\pi}$ $(\overline{H}_*)$ and so it is a normal subgroup of $\mathcal{E}(X)$. With the following example we prove that for a general topological space we cannot expect the same result.

\begin{ex}\label{ex:nonilpotente} Applying the construction obtained in the proof of Theorem \ref{thm:homologygroupsAutE} we can get a topological space $X$ such that $Aut(X)$ is trivial, $\mathcal{E}(X)= S_3$ and $X$ is weak homotopy equivalent to a circle, where $S_3$ denotes the symmetric group on a set of $3$ elements. In Figure \ref{fig:nonilpotente} we present the Hasse diagram of $X$. By construction, every self-homotopy equivalence of $X$ fixes the blue, black, green and orange parts of the Hasse diagram. On the other hand, the only part that contributes to the homotopy groups or homology groups is the black one. Therefore we can deduce that $\mathcal{E}_\#(X)=\mathcal{E}_*(X)=\mathcal{E}(X)=S_3$, which implies that $\mathcal{E}_\#(X)$ and $\mathcal{E}_*(X)$ are not nilpotent groups.

\begin{figure}[h]
\centering
\includegraphics[scale=0.8]{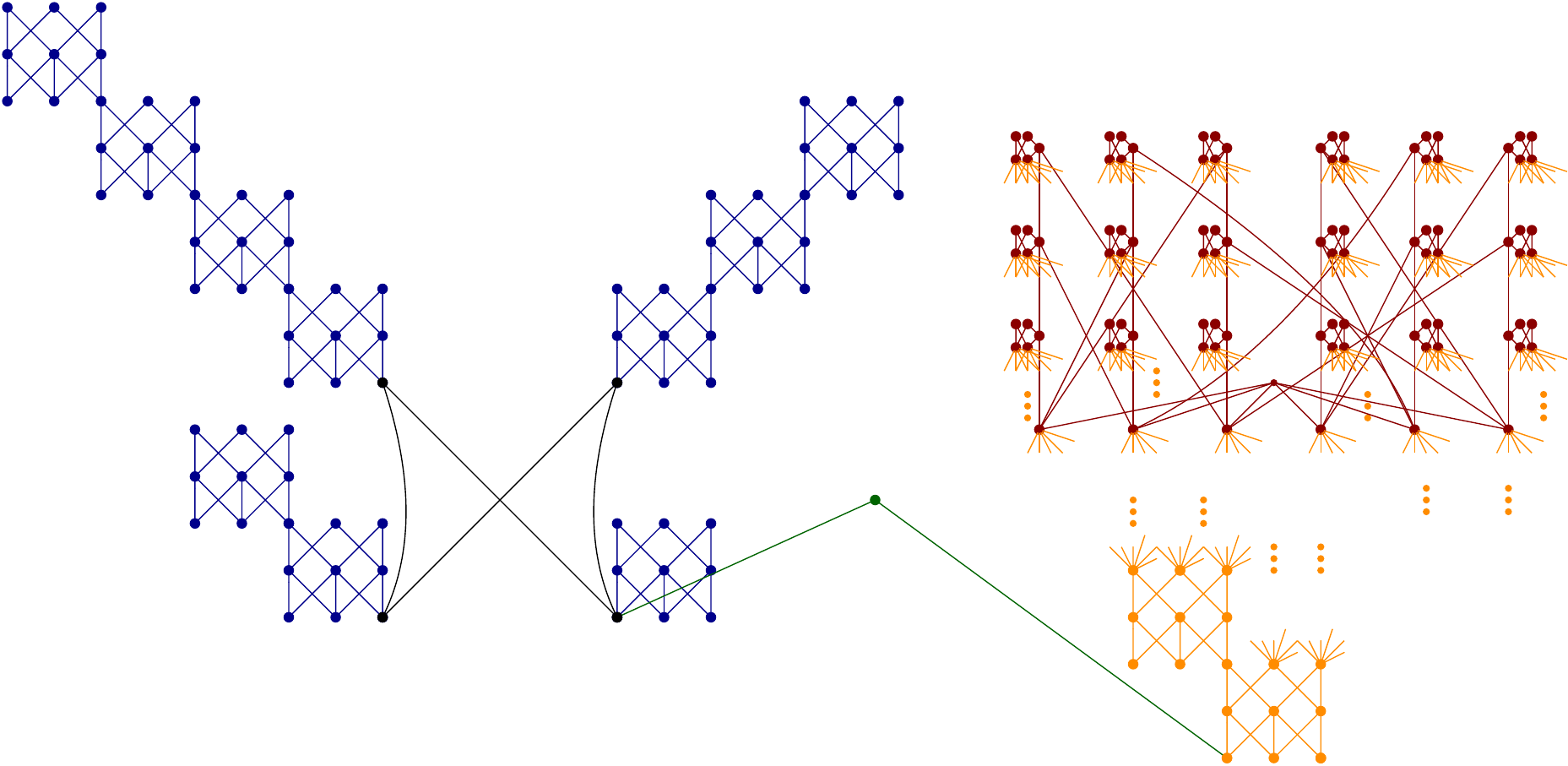}
\caption{Hasse diagram of $X$.}\label{fig:nonilpotente}

\end{figure}
\end{ex}

\begin{rem} It is not difficult to show that the topological spaces obtained in the proof of Theorem \ref{thm:homologygroupsAutE} satisfy that $\mathcal{E}_*(X)=\mathcal{E}_\#(X)=\mathcal{E}(X)$. Then it is easy to find more examples such as Example \ref{ex:nonilpotente}.
\end{rem}

\textbf{Acknowledgment.} We thank Jesús Antonio Álvarez López for posing the question that led to one of the main results of this paper. We wish to express our gratitude to Jonathan A. Barmak for his useful comments and suggestions that substantially improved this manuscript. The authors gratefully acknowledges the many helpful corrections of Andrei Martínez Finkelshtein, Elena Castilla and Jaime J. Sánchez-Gabites during the revision of the paper.

\bibliography{bibliografia}

\begin{thebibliography}{10}

\bibitem{alexandroff1937diskrete}
P.~S. Alexandroff.
\newblock Diskrete {R}äume.
\newblock {\em Mathematiceskii Sbornik (N.S.)}, 2(3):501--519, 1937.

\bibitem{babai1980finite}
L.~Babai.
\newblock Finite digraphs with given regular automorphism groups.
\newblock {\em Period. Math. Hung.}, 1:257,270, 1980.

\bibitem{barmak2011algebraic}
J.~A. Barmak.
\newblock {\em Algebraic topology of finite topological spaces and
  applications}, volume 2032.
\newblock Springer, 2011.

\bibitem{barmak2020automorphism2}
J.~A. Barmak.
\newblock Automorphism groups of finite posets {II}.
\newblock {\em Preprint. arXiv:2008.04997}, 2020.

\bibitem{barmak2009automorphism}
J.~A. Barmak and E.~G. Minian.
\newblock Automorphism groups of finite posets.
\newblock {\em Discrete Math.}, 309(10):3424--3426, 2009.

\bibitem{birkhoff1936order}
G.~Birkhoff.
\newblock On groups of automorphisms.
\newblock {\em Rev. Un. Mat. Argentina}, 11:155--157, 1946.

\bibitem{chocano2020topological}
P.~J. Chocano, M.~A. Mor\'on, and F.~R. Ruiz~del Portal.
\newblock Topological realizations of groups in alexandroff spaces.
\newblock {\em Rev. R. Acad. Cien. Serie A. Mat.}, DOI:
  10.1007/s13398-020-00964-7, 2021.

\bibitem{cianci2020smallest}
N.~Cianci and M.~Ottina.
\newblock Smallest weakly contractible non-contractible topological spaces.
\newblock {\em Proc. Edinburgh Math. Soc.}, 63(1):263--274, 2020.

\bibitem{costoya2014every}
C.~Costoya and A.~Viruel.
\newblock Every finite group is the group of self-homotopy equivalences of an
  elliptic space.
\newblock {\em Acta Math.}, 213(1):49--62, 2014.

\bibitem{costoya2020primer}
C.~Costoya and A.~Viruel.
\newblock A primer on the group of self-homotopy equivalences: a {R}ational
  {H}omotopy {T}heory approach.
\newblock {\em Graduate J. Math.}, 5(1):76..87, 2020.

\bibitem{hatcher2000algebraic}
A.~Hatcher.
\newblock {\em Algebraic topology}.
\newblock Cambridge Univ. Press, Cambridge, 2000.

\bibitem{kahn1976realization}
D.~W. Kahn.
\newblock Realization problems for the group of homotopy classes of
  self-equivalences.
\newblock {\em Math. Ann.}, 220(1):37--46, 1976.

\bibitem{kukiela2010homotopy}
M.~J. Kukie{\l}a.
\newblock On homotopy types of {A}lexandroff spaces.
\newblock {\em Order}, 27(1):9--21, 2010.

\bibitem{may1966finite}
J.~P. May.
\newblock Finite spaces and larger contexts.
\newblock {\em Unpublished book}, 2016.

\bibitem{mccord1966singular}
M.~C. McCord.
\newblock Singular homology groups and homotopy groups of finite topological
  spaces.
\newblock {\em Duke Math. J.}, 33(3):465--474, 1966.

\bibitem{rival1976afixed}
I.~Rival.
\newblock A fixed point theorem for finite partially ordered sets.
\newblock {\em J. Combin. Theory A 21}, pages 309--318, 1976.

\bibitem{stong1966finite}
R.~E. Stong.
\newblock Finite topological spaces.
\newblock {\em Trans. Amer. Math. Soc.}, 123(2):325--340, 1966.

\bibitem{thornton1972spaces}
M.~C. Thornton.
\newblock Spaces with given homeomorphism groups.
\newblock {\em Proc. Amer. Math. Soc.}, 33(1):127--131, 1972.

\bibitem{whittaker1963onIsomorphic}
J.~V. Whittaker.
\newblock On {I}somorphic {G}roups and {H}omeomorphic {S}paces.
\newblock {\em Ann. of Math.}, 78(1):74--91, 1963.

\end{thebibliography}
\bibliographystyle{plain}

\newcommand{\Addresses}{{% additional braces for segregating \footnotesize
  \bigskip
  \footnotesize

  \textsc{ P.J. Chocano, Departamento de \'Algebra, Geometr\'ia y Topolog\'ia, Universidad Complutense de Madrid, Plaza de Ciencias 3, 28040 Madrid, Spain}\par\nopagebreak
  \textit{E-mail address}:\texttt{pedrocho@ucm.es}

  \medskip

\textsc{ M. A. Mor\'on,  Departamento de \'Algebra, Geometr\'ia y Topolog\'ia, Universidad Complutense de Madrid and Instituto de
Matematica Interdisciplinar, Plaza de Ciencias 3, 28040 Madrid, Spain}\par\nopagebreak
  \textit{E-mail address}: \texttt{ma\_moron@mat.ucm.es}

  \medskip

\textsc{ F. R. Ruiz del Portal,  Departamento de \'Algebra, Geometr\'ia y Topolog\'ia, Universidad Complutense de Madrid and Instituto de
Matematica Interdisciplinar
, Plaza de Ciencias 3, 28040 Madrid, Spain}\par\nopagebreak
  \textit{E-mail address}: \texttt{R\_Portal@mat.ucm.es}

}}

\Addresses

\end{document}